\documentclass[10pt,a4paper]{article}

\usepackage{amsmath,amsfonts,amsthm,amssymb,amscd, appendix, verbatim}
\usepackage{xpatch}
\usepackage{hyperref}
\usepackage[utf8]{inputenc}
\usepackage{setspace, color}
\usepackage{array}
\usepackage{cite}

\newcommand{\caixa}{\hglue15.7cm$\square$\vspace{5mm}}

\newtheorem{theorem}{Theorem}[section]
\newtheorem{corollary}[theorem]{Corollary}

\newtheorem{lemma}[theorem]{Lemma}

\theoremstyle{definition}

\newtheorem{remark}[theorem]{Remark}

\makeatletter
   \xpatchcmd{\@thm}{\fontseries\mddefault\upshape}{}{}{} % same font as thm-header
\makeatother

\usepackage[margin=1in]{geometry}

\title{Exponential blow-up of mild solutions to the fractional Boussinesq equations in the Gevrey class}
\author{
Wilberclay G. Melo\footnote{Departamento de Matemática, Universidade Federal de Sergipe, São Cristóvão SE 49100-000,
  Brazil, e-mail: wilberclay@gmail.com}\,,
Cilon Perusato\footnote{Departamento de Matemática, Universidade Federal de Pernambuco, Recife PE 50670-901,
  Brazil, e-mail: cilonperusato@gmail.com. This author was partially supported by CNPq grant No. 10444/2022-5.}\,,
 and Thyago S. R. Santos\footnote{Departamento de Matemática,  Instituto de Matemática, Estatística e Computação Científica, Universidade Estadual de Campinas, Campinas SP 13083-859, Brazil e-mail: thyagosr@unicamp.br. This author was partially supported by FAPESP grant No. 2024/15587-1.} \,}

\date{}

\begin{document}

\maketitle

\begin{abstract}
\noindent This work establishes conditions for the existence and uniqueness of local mild solutions to the Boussinesq equations with fractional dissipations in Sobolev-Gevrey spaces. We prove that a unique mild solution exists in an appropriate Sobolev-Gevrey class and analyze its behavior up to the maximal time of existence. In particular, we derive quantitative lower bounds describing how the norm of the solution must blow up as it approaches a finite maximal time. As a corollary, we deduce that the solution exhibits exponential growth.

% This paper focus on presenting some assumptions in order to obtain the existence of a unique local mild solution for the Boussinesq equations with fractional dissipations $\alpha>\frac{1}{2}$ and $\beta>\frac{1}{2}$ in Sobolev-Gevrey spaces. In addition, this work establishes some properties at potential blow-up times for this same solution. More precisely, the authors prove that there is a unique local solution $(u,\theta)\in C([0,T];\dot{H}_{a,\sigma}^s)$ (with $a>0$, $\sigma>1$, $\alpha>\frac{1}{2}$, $\beta>\frac{1}{2}$, $T>0$ and $0\leq s<\frac{3}{2}$) for the Boussinesq equations and this solution, in the maximal time interval $[0,T^*)$ (with $T^*$ finite), satisfies for $\alpha\geq1$ and $\beta\geq1$, for example, the following inequalities:
% $$\displaystyle \|(u,\theta)(t)\|_{\dot{H}_{\frac{a}{(\sqrt{\sigma})^n},\sigma}^s}^{\frac{2\alpha}{2\alpha-1}}+\|(u,\theta)(t)\|_{\dot{H}_{\frac{a}{(\sqrt{\sigma})^n},\sigma}^s}^{\frac{2\beta}{2\beta-1}}\geq C_{1}[e^{C_{2}(T^{*}-t)}-1]^{-1},$$
% and also, by assuming $\beta=\alpha,$ one has
% $$\displaystyle
% \frac{a^{\sigma_0+\frac{1}{2}}C_3}{[e^{C_4(T^*-t)}-1]^{\frac{(2\alpha-1)[2(s\sigma+\sigma_0)+1]}{6\alpha\sigma}}}
% \exp\left\{\frac{aC_5}{[e^{C_6(T^*-t)}-1]^{\frac{2\alpha-1}{3\alpha\sigma}}}\right\}\leq \|(u,\theta)(t)\|_{\dot{H}_{a,\sigma}^s},$$
% for all $t\in [0,T^*)$, where $n\in \mathbb{N}$, $C_j$ ($j=1,2,3,4,5,6$) is a positive constant and $2\sigma_0$ is the integer part of $2\sigma\mu$ (with $\mu>\frac{3}{2}$ fixed).
\end{abstract}

\textbf{Key words:} {\it Boussinesq equations; existence of local mild  solutions; blow-up criteria; Sobolev-Gevrey spaces.}

\textbf{AMS Mathematics Subject Classification:} 35Q30, 35A01, 76D03, 35B44.

\section{Introduction}

%Magnetohydrodynamics (MHD) describes the motion of electrically conducting fluids, such as liquid metals, salt water, or plasmas, in the presence of magnetic fields. The MHD Equations form a system of equations that combine fluid dynamics with Maxwell's theory of electromagnetism.

\hspace{0.5cm}
In this paper, we consider the following three-dimensional Boussinesq equations:
\begin{equation}\label{micropolar}
\left\{
\begin{array}{l}
u_t
\;\!+\,
u \cdot \nabla u
\,+\,
\nabla \;\!p \:\!\,+\,  \:\!(-\Delta)^\alpha u
\;=\;
\theta e_3, \quad x\in \mathbb{R}^3, \quad t > 0,\\
%
%\mbox{} \vspace{-0.300cm} \\
%
%
%\mbox{} \vspace{-0.300cm} \\
%
\theta_t
\;\!+\,
u \cdot \nabla \theta\,+\, \,(-\Delta)^\beta \theta
\;=\;
0, \quad x\in \mathbb{R}^3, \quad t > 0,\\
%
%\mbox{} \vspace{-0.300cm} \\
%
\mbox{div}\:u  \;=\; 0, \quad x\in \mathbb{R}^3, \quad t > 0,\\
%
%\mbox{} \vspace{-0.300cm} \\
%
u(x,0) \,=\, u_0(x),
\;\,
\theta(x,0) \,=\, \theta_0(x), \quad x\in \mathbb{R}^3,
\end{array}
\right.
\end{equation}
where $u(x,t) = \left(u_{1}(x,t), u_{2}(x,t), u_{3}(x,t)\right) \in \mathbb{R}^3$ denotes the incompressible velocity field, and 
$\theta(x,t) \in \mathbb{R}$ represents the temperature of the fluid. The parameters $\alpha$ and $\beta$ belong to 
$(\tfrac{1}{2},\infty)$, and we set $e_{3} = (0,0,1)$. The initial velocity field $u_{0}$ in \eqref{micropolar} is assumed to be divergence-free, that is,
$$
\operatorname{div} u_{0} = 0,
$$
and we denote by $(-\Delta)^{\gamma}$ the fractional Laplacian, defined for a suitable function $f$ by
$$
\mathcal{F}[(-\Delta)^{\gamma} f](\xi) = |\xi|^{2\gamma} \widehat{f}(\xi) .
$$

Observe that, in the case $\theta \equiv 0$ and $\alpha = 1$, the system \eqref{micropolar} reduces to the classical incompressible Navier-Stokes equations
\begin{equation}\label{NS}
\left\{
\begin{array}{l}
u_t
\;\!+\,
u \cdot \nabla u
\,+\,
\nabla \;\!p  
\;=\;
\Delta u, \quad x\in \mathbb{R}^3, \quad t > 0,\\

\mbox{div}\:u  \;=\; 0, \quad x\in \mathbb{R}^3, \quad t > 0.

\end{array}
\right.
\end{equation}

The Boussinesq equations (\ref{micropolar}) are simplified yet powerful models widely used in the study of oceanic and atmospheric dynamics. By incorporating buoyancy effects while filtering out sound waves, they provide a mathematically tractable framework for describing stratified fluid flows under the influence of gravity. These equations also arise in several other areas of Physics, including thermal convection, geophysical fluid dynamics, and plasma modeling; see, for instance, \cite{MR2245751,MR2540168} for more details. 

It is worth noting, however, that although \eqref{micropolar} with fractional dissipation may initially appear to be a purely mathematical generalization, there are geophysical scenarios in which the Boussinesq equations  with a fractional Laplacian (\ref{micropolar}) naturally arise. A typical example occurs in the middle atmosphere, where upward-travelling flows experience changes due to variations in atmospheric properties, even though the incompressibility and Boussinesq approximations remain valid. In this regime, the effects of kinematic and thermal diffusion are attenuated by the thinning of the atmosphere, an anomalous behavior that can be effectively modeled using a spatial fractional Laplacian. See \cite{MR2379269, Gill1982} for further discussions.

There is an extensive literature, in both two and three dimensions, concerning the well-posedness and qualitative properties of the system \eqref{micropolar} in a variety of functional settings (see, for instance, \cite{MR4884564,MR3503190,MR4795499,MR3808344,MR3008326,MR3759571,cilon1} and the references therein).

In this work, we are concerned with the time evolution of solutions. More precisely, we aim to establish finite-time blow-up criteria for mild solutions of \eqref{micropolar} in Sobolev-Gevrey type spaces. For $a \ge 0$, $\sigma \ge 1$, and $s \in \mathbb{R}$, the (homogeneous) Sobolev-Gevrey space is defined by
$$
\dot{H}_{a,\sigma}^s=\dot{H}_{a,\sigma}^s(\mathbb{R}^3)
    = \left\{ f \in \mathcal{S}' : \widehat{f} \in L^1_{\operatorname{loc}}\,\, \text{and} \,\, \ 
       \|f\|_{\dot{H}_{a,\sigma}^s}:=\Big[\int_{\mathbb{R}^3} |\xi|^{2s} e^{2a |\xi|^{1/\sigma}} |\widehat{f}(\xi)|^2\, d\xi\Big]^{\frac{1}{2}} < \infty
      \right\},
$$
where $\widehat{f}$ denotes the Fourier transform
$$
\mathcal{F}(f)(\xi)=\widehat{f}(\xi):=\int_{\mathbb{R}^3}e^{-i\xi\cdot x}f(x)\,dx.
$$
Essentially, $\dot{H}_{a,\sigma}^s \equiv L^{2}\!(|\xi|^{s} e^{a|\xi|^{1/\sigma}}\, d\xi)$ and, in particular, $\dot{H}_{0,\sigma}^s \equiv \dot{H}^s$, \textit{i.e.}, the classical (homogeneous) Sobolev space.  This class of functions plays a crucial role because, according to  Paley--Wiener Theorem (see \cite{HormanderALPDO1}, Chapter~7, for further details), for $\sigma=1$ a function $ f $ belongs to the space $ \dot{H}_{a,1}^s(\mathbb{R}^d) $ if, and only if, it admits a holomorphic extension $ F $ to the strip

$$
S_a = \{ x + i y  \in \mathbb{C}^d: x, y \in \mathbb{R}^d,\ |y| < a \},
$$
such that
$$
    \sup_{|y| < a} \| F(x + i y) \|_{\dot{H}^s} < \infty.
$$
In other words, the parameter $ a \geq 0 $ determines the width of the complex strip to which 
functions in $ \dot{H}_{a,1}^s(\mathbb{R}^d) $ can be analytically extended. For $\sigma>1$, the regularity falls into a non-analytic Gevrey regime, meaning that derivatives grow in a controlled but non-analytic manner, which is equivalently reflected by a subexponential decay at high frequencies.

Consequently, the study of 
differential equations in such analyticity-based function classes has attracted significant attention in recent years. See, for instance, \cite{MR3504420, MR4632081, MR4369830, MR2169876, MR1026858, MR2265624} and references therein for more details on this topic.

The first main result of this work establishes the local well-posedness of the Cauchy problem \eqref{micropolar} in  Sobolev-Gevrey class. Recall that, we say that $(u,\theta)=(u,\theta)(x,t)$ is a mild solution to the Boussinesq equations (\ref{micropolar}) if this application satisfies the associated integral formulations (\ref{wilber15}) and (\ref{wilber16}), which are established via the fractional heat semigroup.

\begin{theorem}\label{teoremaexistenciaB}
Assume that $a>0$, $\sigma>1$, $\alpha>\frac{1}{2}$, $\beta>\frac{1}{2}$ and $0\leq s<\frac{3}{2}$. Let $(u_0,\theta_0)\in \dot{H}_{a,\sigma}^s$ such that $\hbox{\emph{div}}\,u_0=0$. Then,
  there exist an instant $T=T(a,\sigma,s,\alpha,\beta,u_0,\theta_0)>0$ and a unique mild solution $(u,\theta)\in C_T(\dot{H}_{a,\sigma}^s)$ for the Boussinesq equations \emph{(\ref{micropolar})} that satisfies
    \begin{align*}
  &\|(u,\theta)\|_{L^\infty_T(\dot{H}_{a,\sigma}^{s})}
\lesssim \frac{1-T}{T^{1-\frac{1}{2\alpha}}+T^{1-\frac{1}{2\beta}}}.
  \end{align*}
 \end{theorem}
The proof of this result relies on a fixed-point argument. However, for both the classical and the fractional Boussinesq equations, a major difficulty arises in handling the coupling terms $\theta e_{3}$ and $(u \cdot \nabla)\theta$, which demand delicate estimates within the chosen functional framework.

Next, we investigate the qualitative behavior of the local solution provided by Theorem \ref{teoremaexistenciaB}, assuming that its maximal lifespan is finite.

\begin{theorem}\label{teoremaB}
Assume that $a>0$, $\sigma>1$, $\alpha\geq 1, \beta \geq 1, 0\leq s<\frac{3}{2}$ and $n\in \mathbb{N}$. Let $(u_0,\theta_0)\in \dot{H}_{a,\sigma}^s$ such that $\hbox{\emph{div}}\,u_0=0$. Consider that $(u,\theta)\in C([0,T^*);\dot{H}_{a,\sigma}^s)$ is the  solution
for the Boussinesq equations \emph{(\ref{micropolar})} in the maximal time interval $0\leq t < T^*$  given in Theorem \emph{\ref{teoremaexistenciaB}}. If $T^*<\infty$, then the following statements hold:
      \begin{enumerate}
  \item[\textbf{\emph{i)}}]  $\displaystyle \limsup_{t\nearrow T^*} \|(u,\theta)(t)\|_{\dot{H}_{\frac{a}{(\sqrt{\sigma})^{(n-1)}},\sigma}^s}=+\infty$;
  \item[\textbf{\emph{ii)}}] $\displaystyle\int_{t}^{T^*} [\|e^{\frac{a}{\sigma(\sqrt{\sigma})^{(n-1)}}|\cdot|^{\frac{1}{\sigma}}}(\widehat{u},\widehat{\theta})(\tau)\|_{L^1}
^{\frac{2\alpha}{2\alpha-1}}+\|e^{\frac{a}{\sigma(\sqrt{\sigma})^{(n-1)}}|\cdot|^{\frac{1}{\sigma}}}(\widehat{u},\widehat{\theta})(\tau)\|_{L^1}^{\frac{2\beta}{2\beta-1}}]\,d\tau=\infty$;
 \item[\textbf{\emph{iii)}}]  $\displaystyle \|e^{\frac{a}{\sigma(\sqrt{\sigma})^{(n-1)}}|\cdot|^{\frac{1}{\sigma}}}(\widehat{u},\widehat{\theta})(t)\|_{L^1}^{\frac{2\alpha}{2\alpha-1}}
    +\|e^{\frac{a}{\sigma(\sqrt{\sigma})^{(n-1)}}|\cdot|^{\frac{1}{\sigma}}}(\widehat{u},\widehat{\theta})(t)\|_{L^1}^{\frac{2\beta}{2\beta-1}}\geq  (e^{C(T^*-t)}-1)^{-1};$
  \item [\textbf{\emph{iv)}}] $\displaystyle \|(u,\theta)(t)\|_{\dot{H}_{\frac{a}{(\sqrt{\sigma})^n},\sigma}^s}^{\frac{2\alpha}{2\alpha-1}}+
\|(u,\theta)(t)\|_{\dot{H}_{\frac{a}{(\sqrt{\sigma})^n},\sigma}^s}^{\frac{2\beta}{2\beta-1}}\gtrsim(e^{C(T^{*}-t)}-1)^{-1}.$
  \end{enumerate}
for all $t\in [0,T^{*})$ and $C>0$ is a positive constant.
\end{theorem}

 As a corollary, we obtain that the local solutions provided by Theorem \ref{teoremaexistenciaB} exhibit exponential growth, which in turn ensures finite-time blow-up with an explicit exponential rate.

\begin{corollary}\label{corolario}
Assume that $a>0$, $\sigma>1$, $\alpha=\beta\geq 1 \hbox{  and  } 0\leq s<\frac{3}{2}$. Let $(u_0,\theta_0)\in \dot{H}_{a,\sigma}^s\cap L^2$ such that $\hbox{\emph{div}}\,u_0=0$. Consider that $(u,\theta)\in C([0,T^*);\dot{H}_{a,\sigma}^s)$ is the  solution
for the Boussinesq equations \emph{(\ref{micropolar})} in the maximal time interval $0\leq t < T^*$  given in Theorem \emph{\ref{teoremaexistenciaB}}. If $T^*<\infty$, then the following statement holds:
   $$
   \displaystyle
\|(u,\theta)(t)\|_{\dot{H}_{a,\sigma}^s} \gtrsim  {(e^{C_1(T^*-t)}-1)^{\varrho_1}}
\exp\{C_2(e^{C_1(T^*-t)}-1)^{\varrho_2}\},
$$
for all \(t \in [0, T^{*})\), where $\varrho_1 := \frac{(1 - 2\alpha)\,[\,2(s\sigma + \sigma_{0}) + 1\,]}{6\alpha\sigma}<0$, $\varrho_2:=\frac{1-2\alpha}{3\alpha\sigma}<0$ and \(2\sigma_{0}\) denotes the integer part of \(2\sigma\mu\) \emph{(}with \(\mu > \tfrac{3}{2}\)\emph{)}.  
Here, $C_1=C_1(\alpha)$ and $C_2=C_2(a,\alpha,\sigma,s,u_0,\theta_0,T^*)$ are constants. In particular
$$
\lim_{t \nearrow T^*}\|(u,\theta)(t)\|_{\dot{H}_{a,\sigma}^s} = +\infty,
$$
exponentially.
\end{corollary}

\bigskip

\begin{remark}
Let us point out  relevant  improvements presented by our main  results.
\begin{enumerate}
  \item \underline{Solution spaces}: In this text, we have approached Gevrey classes as solution spaces. Thus, we have decided to compare our contributions to some results established by R.~Guterres, W.~G. Melo, N.~F. Rocha and T.~S.~R. Santos \cite{robert} (see also \cite{wilberr2,wilberr3,wilberr4,wilberr5,wilberr6,wilberr7,MR3504420,MR4632081,MR2169876,MR4369830} and references therein), which ones are similar to ours. First of all, Theorem \ref{teoremaexistenciaB} presents weaker regularities for the mild solutions than Theorem $1.1$ obtained in \cite{robert} (check also Theorem $1.4$ in \cite{robert}, which is related to Theorem \ref{teoremaB} and Corollary \ref{corolario}). This fact occurs because we have applied Lemma \ref{lemalorenz} instead of Lemma $2.3$  presented in \cite{robert}, and have worked in a slightly different way as well  (see (\ref{wilber17}) and (\ref{w1}) below for more details). However, this choice led us to not consider the usual homogenous Sobolev spaces (see the proof of Lemma $2.9$ in \cite{lorenz}). In addition, at last, it is worth to notice that $s<\frac{3}{2}$ is our assumption due to the completeness of Sobolev-Gevrey space (see also Lemma \ref{lemalorenz}) and, furthermore, $\alpha,\beta\geq 1$ in Theorem \ref{teoremaB} and Corollary \ref{corolario} because of the use of Lemma \ref{lemanovow1} (it is a technical issue).   
  \item \underline{Boussinesq equations}: The term $\theta e_3$ given in (\ref{micropolar}) is the main mathematical difference between the Boussinesq equations and the generalized MHD equations studied by \cite{robert} (and, consequently, the Navier-Stokes equations (\ref{NS})) in this paper. Because of this term, we have made the use of a different fixed point theorem (see Lemma \ref{pontofixo} and Lemma $2.1$ in \cite{lorenz}). More specifically, we were supposed to add new equalities and estimates  in order to obtain our conclusions (see, for example, (\ref{linearL}), (\ref{linearL1eL2})--(\ref{lwfinal}), (\ref{3}), (\ref{esqueci}), (\ref{wilber19})--(\ref{wilber23}), (\ref{wilber4}), (\ref{wilber24}), (\ref{wilber25}), (\ref{wilber26})--(\ref{wilber9}), (\ref{i)n=1}), (\ref{wilber32}), (\ref{normal2}) and (\ref{wilber33})).
  \item \underline{New lower bounds}: As a consequence of the results discussed in the last item above, we have proved new blow-up criteria for local mild solutions of the Boussinesq equations (\ref{micropolar}) as, for instance, Theorem \ref{teoremaB} \textbf{iii)} and \textbf{iv)}, and Corollary \ref{corolario} (compare these results to Theorem $1.4$ iii), iv) and v) in \cite{robert}). Lastly, it is important to point out that other information has been showed in this work, see Theorem \ref{teoremaexistenciaB} and Theorem \ref{teoremaB} \textbf{i)} and \textbf{ii)}.
\end{enumerate}

\end{remark}

\bigskip

\noindent\textbf{Notations. }For  $(X,\|\cdot\|_{X})$ be a normed space and $I \subset \mathbb{R}$ an interval. We define  
$$
C(I;X)=\{f:I\to X \text{ continuous}\}, \qquad 
\|f\|_{L^{\infty}(I;X)}=\sup_{t\in I}\|f(t)\|_{X},
$$
and for $T>0$, we  write $C_{T}(X)=C([0,T];X)$. For $1 \le p < \infty$, the Bochner space $L^{p}(I;X)$ is equipped with the norm  
$$
\|f\|_{L^{p}(I;X)}=\left(\int_{I}\|f(t)\|_{X}^{p}\,dt\right)^{1/p},
$$
and we denote $L_{T}^{p}(X)=L^{p}([0,T];X)$ with$\footnote{Here, $\|(f,g)\|_{L^1}:=\|f\|_{L^1}+\|g\|_{L^1}$  and $\|(f,g)\|_{\dot{H}_{a,\sigma}^s}:=[\|f\|^{2}_{\dot{H}_{a,\sigma}^s}+\|g\|^{2}_{\dot{H}_{a,\sigma}^s}]^{\frac{1}{2}}.$}$ $1\leq p \leq \infty$. Furthermore, for the functions $f$ and $g=(g_1,g_2,g_3)$, the tensor product is  defined by $f \otimes g := (g_{1}f, g_{2}f, g_{3}f)$.

\bigskip

\noindent\textbf{Organization of the paper.}
The structure of the paper is as follows. 
In Section~\ref{secaonotacoes}, we present the auxiliary lemmas and technical tools required for the proofs of our main results. 
Section~$3$ is devoted to the proofs of Theorems~\ref{teoremaexistenciaB} and~\ref{teoremaB}, as well as Corollary~\ref{corolario}, each of which is addressed in a separate subsection.

\section{Preliminary lemmas }\label{secaonotacoes}

\hspace{0.5cm} This section is devoted to the auxiliary lemmas and tools employed throughout this work. We begin by presenting the preliminary lemmas that play a central role in the proofs of our main results.

First, we list the lemmas that will be used in the proof of Theorem~\ref{teoremaexistenciaB}.

	\begin{lemma}[see \cite{lorenz}]\label{calculo}
		Let $a,b>0$. Then, $\displaystyle\lambda^ae^{-b\lambda}\leq a^a(eb)^{-a}$ for all $\lambda >0.$
	\end{lemma}
	\begin{proof}
		For  more details, see Lemma $2.10$ in \cite{lorenz} (and references therein).
\end{proof}

\begin{lemma}[see \cite{lorenz}]\label{lemalorenz}
	Let $a>0$, $\sigma> 1$ and $s\in[0,\frac{3}{2})$. Then, there exists a positive constant $C_{a,\sigma,s}$ such that, for all $f,g\in \dot{H}_{a,\sigma}^{s}(\mathbb{R}^3)$, we have
	$$\|fg\|_{\dot{H}_{a,\sigma}^{s}}\leq C_{a,\sigma,s}	\|f\|_{\dot{H}_{a,\sigma}^{s}}\|g\|_{\dot{H}_{a,\sigma}^{s}}.$$
\end{lemma}
\begin{proof}
	For  more details, see Lemma $2.9$ in \cite{lorenz} (and references therein).
\end{proof}

\begin{lemma}[see \cite{pontofixo}]\label{pontofixo}
		Let $(X,\|\cdot\|)$ be a Banach space, $L:X\rightarrow X$ a continuous linear operator and
 $B:X\times X\rightarrow X$ a continuous bilinear operator, i.e., there exist positive constants $C_1$ and $C_2$ such that
 $$\|L(x)\|\leq C_1 \|x\|,\quad \|B(x,y)\|\leq C_2 \|x\|\|y\|,\quad \forall x,y\in X.$$
 Then, for each $C_1\in (0,1)$ and $x_0\in X$ that satisfy $4C_2\|x_0\| <(1-C_1)^2$, one has that the equation 
 $$a=x_0+B(a,a)+L(a), \quad a\in X,$$
 admits solution $x\in X$. Moreover, $x$ satisfies the inequality $\|x\|\leq \frac{2\|x_0\|}{1-C_1}$ and this is the only one such that $\|x\|\leq  \frac{1-C_1}{2C_2}.$
	\end{lemma}
	\begin{proof}
		For  more details, see Lemma $5$ in \cite{pontofixo} (and references therein).
\end{proof}

Now, we shall list the preliminary lemmas that will be used in the proof of Theorem \ref{teoremaB}. 

\begin{lemma}[see \cite{robert}]\label{lemanovow1}
Let $a\geq0$, $\sigma\geq1$, $s\in\mathbb{R}$ and $\theta\geq1$. The following inequality holds:
$$\|f\|_{\dot{H}_{a,\sigma}^{s+1}}\leq \|f\|_{\dot{H}_{a,\sigma}^{s}}^{1-\frac{1}{\theta}}\|f\|_{\dot{H}_{a,\sigma}^{s+\theta}}^{\frac{1}{\theta}}.$$
  \end{lemma}
\begin{proof}
For more details, see Lemma $2.1$ in \cite{robert}.
\end{proof}

\begin{lemma}[see \cite{robert}]\label{lemanovow2}
Let $a>0$, $\sigma>1$, $s\in[-1, \frac{3}{2})$, $\alpha\geq1$ and $\beta\geq1$. For every $f\in \dot{H}_{a,\sigma}^s\cap \dot{H}_{a,\sigma}^{s+\alpha}$ and $g\in \dot{H}_{a,\sigma}^s\cap \dot{H}_{a,\sigma}^{s+\beta}$, we have that $fg \in \dot{H}_{a,\sigma}^{s+1}$. More precisely, we have
\begin{enumerate}
  \item[\emph{i)}] $\displaystyle\|fg\|_{\dot{H}_{a,\sigma}^{s+1}}\leq C_{s}[\|e^{\frac{a}{\sigma}|\cdot|^{\frac{1}{\sigma}}}\widehat{f}\|_{L^1}\|g\|_{\dot{H}_{a,\sigma}^{s}}^{1-\frac{1}{\beta}}\|g\|_{\dot{H}_{a,\sigma}^{s+\beta}}^{\frac{1}{\beta}}
+\|e^{\frac{a}{\sigma}|\cdot|^{\frac{1}{\sigma}}}\widehat{g}\|_{L^1}\|f\|_{\dot{H}_{a,\sigma}^{s}}^{1-\frac{1}{\alpha}}\|f\|_{\dot{H}_{a,\sigma}^{s+\alpha}}^{\frac{1}{\alpha}}];$
  \item[\emph{ii)}] $\displaystyle\|fg\|_{\dot{H}_{a,\sigma}^{s+1}}\leq C_{a,\sigma,s}[\|f\|_{\dot{H}_{a,\sigma}^{s}}\|g\|_{\dot{H}_{a,\sigma}^{s}}^{1-\frac{1}{\beta}}\|g\|_{\dot{H}_{a,\sigma}^{s+\beta}}^{\frac{1}{\beta}}
+\|g\|_{\dot{H}_{a,\sigma}^{s}}\|f\|_{\dot{H}_{a,\sigma}^{s}}^{1-\frac{1}{\alpha}}\|f\|_{\dot{H}_{a,\sigma}^{s+\alpha}}^{\frac{1}{\alpha}}].
$
\end{enumerate}

\end{lemma}
\begin{proof}
For more details, see Lemma $2.2$ in \cite{robert}.
\end{proof}

	\begin{lemma}[see \cite{Be14}]\label{benameur}
		Let $a\geq0$ and $\sigma\geq 1$. Then, 
$$\displaystyle e^{a|\xi|^\frac{1}{\sigma}}\leq e^{a|\xi-\eta|^\frac{1}{\sigma}}e^{a|\eta|^\frac{1}{\sigma}},\quad \forall\xi,\eta\in \mathbb{R}^3.$$
	\end{lemma}
	\begin{proof}
		For more  details, see the inequality ($17$) in \cite{Be14}.
\end{proof}

Lastly, we shall present the  lemmas that will be used in the proof of Corollary \ref{corolario}.

%\begin{lemma}[see \cite{Be16,r2}]\label{lemaversaocheminH}
%	Let $a\geq0$, $\sigma\geq 1$ and $(s_1,s_2)\in \mathbb{R}^2$ such that $s_1<\frac{3}{2}$ and $s_1+s_2>0$. Then, there exists a positive constant $C_{s_1,s_2}$ such that, for all $f,g\in %\dot{H}_{a,\sigma}^{s_1}(\mathbb{R}^3)\cap\dot{H}_{a,\sigma}^{s_2}(\mathbb{R}^3)$, we have
	%$$\|fg\|_{\dot{H}_{a,\sigma}^{s_1+s_2-\frac{3}{2}}(\mathbb{R}^3)}\leq C_{s_1,s_2}
%	\Big ( \|f\|_{\dot{H}_{a,\sigma}^{s_1}(\mathbb{R}^3)}\|g\|_{\dot{H}_{a,\sigma}^{s_2}(\mathbb{R}^3)}+\|f\|_{\dot{H}_{a,\sigma}^{s_2}(\mathbb{R}^3)}\|g\|_{\dot{H}_{a,\sigma}^{s_1}(\mathbb{R}^3)} %\Big ) .$$
%	If $s_1<\frac{3}{2}$, $s_2<\frac{3}{2}$ and
%	$s_1+s_2>0$, then there is a positive constant $C_{s_1,s_2}$ such that
%	$$\|fg\|_{\dot{H}_{a,\sigma}^{s_1+s_2-\frac{3}{2}}(\mathbb{R}^3)}\leq C_{s_1,s_2}\|f\|_{\dot{H}_{a,\sigma}^{s_1}(\mathbb{R}^3)}\|g\|_{\dot{H}_{a,\sigma}^{s_2}(\mathbb{R}^3)}.$$
%\end{lemma}
%\begin{proof}
%	For  details, see Lemma $2.2$ in \cite{Be16} and Lemma $2.8$ in  \cite{r2}.
%\end{proof}

	\begin{lemma}[see \cite{Be14}]\label{benameur1}
		Let $\delta>\frac{3}{2}$ and $f\in \dot{H}^{\delta}\cap L^2$. Then, there is $C_\delta>0$ such that
$$\|\widehat{f}\|_{L^1}\leq C_\delta \|f\|_{L^2}^{1-\frac{3}{2\delta}}\|f\|_{\dot{H}^{\delta}}^{\frac{3}{2\delta}}.$$
Moreover, for each $\delta_0>\frac{3}{2}$ there is a positive constant $C_{\delta_0}$ such that $C_\delta\leq C_{\delta_0}$ for all $\delta\geq \delta_0.$
	\end{lemma}
	\begin{proof}
		For more details, see Lemma $2.1$ in \cite{Be14}.
\end{proof}

\begin{lemma}[see \cite{lorenz}]\label{lemalorenz1}
	Let $a>0$, $\sigma\geq 1$, $s\in[0,\frac{3}{2})$ and $\delta\geq \frac{3}{2}$. Then, there exists a positive constant $C_{a,\sigma,s,\delta}$ such that
	$$\|f\|_{\dot{H}^{\delta}}\leq C_{a,\sigma,s,\delta}	\|f\|_{\dot{H}_{a,\sigma}^{s}},\quad \forall f\in \dot{H}_{a,\sigma}^{s}.$$
\end{lemma}
\begin{proof}
	For  more details, see Lemma $2.4$ in \cite{lorenz}.
\end{proof}

\section{Proof of the main results:}

This section presents the proofs of Theorems \ref{teoremaexistenciaB} and \ref{teoremaB}, and Corollary \ref{corolario}. More precisely, we shall split this presentation into three subsections, which will contain the  arguments that establish  the veracity of each one of our  main results.

\subsection{Proof of Theorem \ref{teoremaexistenciaB}:}\label{secaoteoremaexistenciaB11}

First of all, apply  Helmholtz's projector $\mathbb{P}$, defined via Fourier transform (see \cite{PN} for more details)
$$
\mathcal{F}[\mathbb{P}f](\xi) :=  \widehat{f}(\xi) - \frac{\widehat{f}(\xi)\cdot \xi}{|\xi|^2}\, \xi,
$$
the heat semigroup $e^{-(t-\tau)(-\Delta)^{\alpha}}$ to the first equation in (\ref{micropolar}), and integrate the result obtained over $[0,t]$ to obtain
\begin{align}\label{wilber15}
u(t)= e^{- t(-\Delta)^{\alpha}}u_0 - \int_{0}^te^{- (t-\tau)(-\Delta)^{\alpha}} \mathbb{P}(u\cdot\nabla u)\,d\tau + \int_{0}^te^{- (t-\tau)(-\Delta)^{\alpha}} \mathbb{P}(\theta e_3)\,d\tau,\quad\forall t>0.
\end{align}
%It is known that this operator satisfies
%\begin{align}\label{transformadadeP}
%\mathcal{F}[P(f)](\xi)=\widehat{f}(\xi)-\frac{\widehat{f}(\xi)\cdot \xi}{|\xi|^2}\xi,\quad\forall \xi\in \mathbb{R}^3.
%\end{align}
Similarly, by using the heat semigroup $e^{- (t-\tau)(-\Delta)^{\beta}}$ in the second equation of (\ref{micropolar}), and integrating over $[0,t]$, we deduce
\begin{align}\label{wilber16}
\theta(t)= e^{- t(-\Delta)^{\beta}}\theta_0 - \int_{0}^te^{-  (t-\tau)(-\Delta)^{\beta}} [u\cdot \nabla\theta]\,d\tau ,\quad\forall t>0.
\end{align}
Hence, we are able to write the following equality:
\begin{align}\label{estimativaprojecao3}
(u,\theta)(t)&= (e^{- t(-\Delta)^{\alpha}}u_0,e^{- t (-\Delta)^{\beta}}\theta_0) + B((u,\theta),(u,\theta))(t)+L(u,\theta)(t),\quad\forall t>0.
\end{align}
where
\begin{align}\label{bilinearB}
B((w,v),(\gamma,\phi))(t)=(B_1((w,v),(\gamma,\phi)),B_2((w,v),(\gamma,\phi)))(t),\quad\forall t>0,
\end{align}
and also
\begin{align}\label{linearL}
L(w,v)(t)=(L_1(w,v)(t),L_2(w,v)(t)),\quad\forall t>0,
\end{align}
with
\begin{align}\label{bilinearB1}
B_1((w,v),(\gamma,\phi))(t)= - \int_{0}^te^{- (t-\tau)(-\Delta)^{\alpha}} \mathbb{P}(w\cdot\nabla \gamma)\,d\tau,
\end{align}
\begin{align}\label{bilinearB2}
B_2((w,v),(\gamma,\phi))(t)= - \int_{0}^te^{- (t-\tau)(-\Delta)^{\beta}} [w\cdot \nabla \phi]\,d\tau,
\end{align}
\begin{align}\label{linearL1eL2}
L_1(w,v)(t)= \int_{0}^te^{- (t-\tau)(-\Delta)^{\alpha}} \mathbb{P}[ve_3]\,d\tau \quad \hbox{    and    }\quad  L_2(w,v)(t)= 0,
\end{align}
for all $w,v,\gamma,\phi\in C_T(\dot{H}_{a,\sigma}^s)$ ($T>0$ will be revealed below). Denote $X= [C_T (\dot{H}_{a,\sigma}^{s})]^3$ $\times C_T (\dot{H}_{a,\sigma}^{s}) (\equiv C_T (\dot{H}_{a,\sigma}^{s})$ throughout this work)$\footnote{Here, $\|(f,g)\|_{X}:=[\|f\|^2_{L^\infty_T(\dot{H}_{a,\sigma}^s)}+\|g\|^2_{L^\infty_T(\dot{H}_{a,\sigma}^s)}]^{\frac{1}{2}}$}$ and notice that $B$ and $L$ are  bilinear and linear operators on $X\times X$ and $X$, respectively. 

We start proving that $L$ is a continuous operator. Observe that (\ref{linearL1eL2})  implies that
\begin{align}\label{wilber18}
\nonumber&\displaystyle\|L_1(w,v)(t)\|_{\dot{H}_{a,\sigma}^s}\leq  \int_0^t\Big(\int_{\mathbb{R}^3} e^{-2(t-\tau)|\xi|^{2\alpha}}|\xi|^{2s} e^{2a|\xi|^{\frac{1}{\sigma}}} |\mathcal{F}[\mathbb{P}( ve_3)](\xi)|^2\,d\xi\Big)^{\frac{1}{2}}\,d\tau\\
\nonumber&\leq\int_0^t\Big(\int_{\mathbb{R}^3} |\xi|^{2s} e^{2a|\xi|^{\frac{1}{\sigma}}} |\mathcal{F}(v)(\xi)|^2\,d\xi\Big)^{\frac{1}{2}}\,d\tau\\
&\leq T \|v\|_{L^{\infty}_T(\dot{H}_{a,\sigma}^{s})},
\end{align}
for all $t\in[0,T]$. As a consequence, by applying (\ref{linearL}) and (\ref{linearL1eL2}), one concludes
\begin{align}\label{lwfinal}
\displaystyle\|L(w,v)\|_{X}\leq  T \|(w,v)\|_{X},\quad\forall (w,v)\in X.
\end{align}
This means that $L$ is a continuous operator.

Now, let us show that $B$ is also a continuous operator. In fact, by using (\ref{bilinearB1}) and   $\hbox{Lemma \ref{calculo}}$, we have
\begin{align}\label{wilber17}
\nonumber\displaystyle\|B_1((w,v),(\gamma,\phi))(t)\|_{\dot{H}_{a,\sigma}^s}&\leq  \int_0^t\Big(\int_{\mathbb{R}^3} |\xi|^{2}e^{-2(t-\tau)|\xi|^{2\alpha}}|\xi|^{2s} e^{2a|\xi|^{\frac{1}{\sigma}}} |\mathcal{F}(\gamma\otimes w)(\xi)|^2\,d\xi\Big)^{\frac{1}{2}}\,d\tau\\
\nonumber&\leq C_{\alpha} \int_0^t (t-\tau)^{-\frac{1}{2\alpha}}\Big(\int_{\mathbb{R}^3} |\xi|^{2s} e^{2a|\xi|^{\frac{1}{\sigma}}} |\mathcal{F}(\gamma\otimes w)(\xi)|^2\,d\xi\Big)^{\frac{1}{2}}\,d\tau\\
&\leq C_{\alpha} \int_0^t (t-\tau)^{-\frac{1}{2\alpha}} \|(\gamma\otimes w)(\tau)\|_{\dot{H}_{a,\sigma}^{s}}\,d\tau.
\end{align}
By the use of  Lemma \ref{lemalorenz} (recall that $a>0,\sigma>1$ and $s\in[0,\frac{3}{2})$), we can write
\begin{align}
\nonumber\displaystyle\|B_1((w,v),(\gamma,\phi))(t)\|_{\dot{H}_{a,\sigma}^s}
%&\leq\int_0^t\frac{C_{s,\mu,\alpha}}{(t-\tau)^{\frac{5-2s}{4\alpha}}}\|[w\otimes \gamma- \phi \otimes v](\tau)\|_{\dot{H}_{a,\sigma}^{2s-\frac{3}{2}}}\,d\tau\\
&\leq C_{a,\sigma,s,\alpha}\|\gamma\|_{L_T^{\infty}(\dot{H}_{a,\sigma}^{s})} \|w\|_{L_T^{\infty}(\dot{H}_{a,\sigma}^{s})}\int_0^t (t-\tau)^{-\frac{1}{2\alpha}} d\tau\\
\label{w1}&\leq C_{a,\sigma,s,\alpha}T^{1-\frac{1}{2\alpha}}\|(w,v)\|_{X} \|(\gamma,\phi)\|_{X},
\end{align}
for all $t\in [0,T]$ (since $\alpha>\frac{1}{2}$). By applying similar arguments, one infers
\begin{align}\label{w2}
\|B_2((w,v),(\gamma,\phi))(t)\|_{\dot{H}_{a,\sigma}^s}\leq C_{a,\sigma,s,\beta}T^{1-\frac{1}{2\beta}}\|(w,v)\|_{X} \|(\gamma,\phi)\|_{X},\
\end{align}
for all $t\in [0,T]$ (see (\ref{bilinearB2}) and recall that $\beta>\frac{1}{2}$). Observing  (\ref{bilinearB}), (\ref{w1}) and (\ref{w2}), we reach the following inequality:
\begin{align}\label{w5}
\|B((w,v),(\gamma,\phi))\|_{X\times X}&\leq C_{a,\sigma,s,\alpha,\beta}[T^{1-\frac{1}{2\alpha}}+T^{1-\frac{1}{2\beta}}]\|(w,v)\|_{X} \|(\gamma,\phi)\|_{X},\
\end{align}
for all $((w,v), (\gamma,\phi))\in X\times X$. This means that $B$ is a continuous operator.

At last, notice that 
\begin{align}\label{mudanca1}
\displaystyle\|e^{- t(-\Delta)^{\alpha}} u_0\|^2_{\dot{H}_{a,\sigma}^s}&=\int_{\mathbb{R}^3} e^{-2 t|\xi|^{2\alpha}}|\xi|^{2s} e^{2a|\xi|^{\frac{1}{\sigma}}} |\widehat{u}_0(\xi)|^2\,d\xi\leq  \|u_0\|_{\dot{H}_{a,\sigma}^{s}}^2,\quad\forall t\in [0,T],
\end{align}
and also that
\begin{align}\label{mudanca2}
\displaystyle\|e^{- t(-\Delta)^{\beta}} \theta_0\|_{\dot{H}_{a,\sigma}^s}\leq  \|\theta_0\|_{\dot{H}_{a,\sigma}^{s}},\quad\forall t\in [0,T].
\end{align}
From (\ref{mudanca1}) and (\ref{mudanca2}), it follows that
\begin{align}\label{dadoincial}
\|(e^{-t(-\Delta)^{\alpha}}u_0,e^{-t (-\Delta)^{\beta}}\theta_0)\|_{X}\leq\|(u_0,\theta_0)\|_{\dot{H}_{a,\sigma}^{s}}.
\end{align}
Thus, we are able to determine $T>0$ as follows: 
$$T< \min\Big\{\Big[\sqrt{8C_{a,\sigma,s,\alpha,\beta}\|(u_0,\theta_0)\|_{\dot{H}_{a,\sigma}^{s}}}+1\Big]^{-\frac{4\alpha}{2\alpha-1}},\Big[\sqrt{8C_{a,\sigma,s,\alpha,\beta}\|(u_0,\theta_0)\|_{\dot{H}_{a,\sigma}^{s}}}+1\Big]^{-\frac{4\beta}{2\beta-1}}\Big\},$$
where $C_{a,\sigma,s,\alpha,\beta}$ is established in (\ref{w5}). Therefore, by  Lemma \ref{pontofixo}, there exists a unique solution $(u,\theta)\in X$ for the equation (\ref{estimativaprojecao3}) (it is enough to take a look at (\ref{lwfinal}), (\ref{w5}) and (\ref{dadoincial})) that satisfies the following inequality:
$$\|(u,\theta)\|_{X}\leq \frac{1-T}{2C_{a,\sigma,s,\alpha,\beta}[T^{1-\frac{1}{2\alpha}}+T^{1-\frac{1}{2\beta}}]}.$$

\caixa

\subsection{Proof of Theorem \ref{teoremaB}}\label{secaoteoremaw2}

In this section, we shall establish the proof of Theorem \ref{teoremaB} by considering the cases $n=1$ and $n=2$. By following the steps given below, it is easy to observe that the other cases are obtained by an induction argument.

\vspace{0.2cm}
\hspace{-0.5cm}\textbf{Proof of Theorem \ref{teoremaB} \textbf{i)} with $n=1$:}\label{subsecaolimitesuperiorw2}
\vspace{0.2cm}

First of all, by using  the first equation of (\ref{micropolar}), we can write the following inequality:
\begin{align}\label{3}
\frac{1}{2}\frac{d}{dt}\|u(t)\|_{\dot{H}_{a,\sigma}^s}^2+\|u(t)\|_{\dot{H}_{a,\sigma}^{s+\alpha}}^2 &\leq
|\langle u, u\cdot\nabla u\rangle_{\dot{H}_{a,\sigma}^s}|+  |\langle u_3,\theta\rangle_{\dot{H}_{a,\sigma}^s} |.
\end{align}
Similarly, by observing the second equation in (\ref{micropolar}), we conclude
\begin{align}\label{4}
\frac{1}{2}\frac{d}{dt}\|\theta(t)\|_{\dot{H}_{a,\sigma}^s}^2+\|\theta(t)\|_{\dot{H}_{a,\sigma}^{s+\beta}}^2 &\leq |\langle \theta, u\cdot\nabla \theta\rangle_{\dot{H}_{a,\sigma}^s}|.
\end{align}
Consequently, from (\ref{3}) and (\ref{4}), it follows that
\begin{align*}
\nonumber\frac{1}{2}\frac{d}{dt}\|(u,\theta)(t)\|_{\dot{H}_{a,\sigma}^s}^2+\|u(t)\|_{\dot{H}_{a,\sigma}^{s+\alpha}}^2+\|\theta(t)\|_{\dot{H}_{a,\sigma}^{s+\beta}}^2 &\leq
 \|u\|_{\dot{H}_{a,\sigma}^{s}}\|u\otimes u\|_{\dot{H}_{a,\sigma}^{s+1}}+ \|\theta\|_{\dot{H}_{a,\sigma}^{s}}\|\theta u\|_{\dot{H}_{a,\sigma}^{s+1}}\\
&\quad+ \|u\|_{\dot{H}_{a,\sigma}^{s}}\|\theta\|_{\dot{H}_{a,\sigma}^{s}}.
\end{align*}
By applying Lemma \ref{lemanovow2} i) and Lemma \ref{lemanovow1}  (since $a>0$, $\sigma>1$, $s\in[0,\frac{3}{2})$ and $\alpha,\beta\geq1$), we conclude
\begin{align}\label{n3}
\nonumber\frac{1}{2}\frac{d}{dt}\|(u,\theta)(t)\|_{\dot{H}_{a,\sigma}^s}^2+\|u(t)\|_{\dot{H}_{a,\sigma}^{s+\alpha}}^2+\|\theta(t)\|_{\dot{H}_{a,\sigma}^{s+\beta}}^2 &\leq C_s \|e^{\frac{a}{\sigma}|\cdot|^{\frac{1}{\sigma}}}(\widehat{u},\widehat{\theta})\|_{L^{1}}
[ \|(u,\theta)\|_{\dot{H}_{a,\sigma}^{s}}^{2-\frac{1}{\alpha}}\|u\|_{\dot{H}_{a,\sigma}^{s+\alpha}}^{\frac{1}{\alpha}}\\
&\quad+
\|(u,\theta)\|_{\dot{H}_{a,\sigma}^{s}}^{2-\frac{1}{\beta}}\|\theta\|_{\dot{H}_{a,\sigma}^{s+\beta}}^{\frac{1}{\beta}}]+ \|(u,\theta)\|_{\dot{H}_{a,\sigma}^{s}}^2.
\end{align}
Hence, apply the elementary Young's inequality and Lemma \ref{lemanovow2} ii) in order to obtain
\begin{align}\label{esqueci}
\frac{d}{dt}\|(u,\theta)(t)\|_{\dot{H}_{a,\sigma}^s}^2+\|u(t)\|_{\dot{H}_{a,\sigma}^{s+\alpha}}^2+\|\theta(t)\|_{\dot{H}_{a,\sigma}^{s+\beta}}^2 &\leq C_{a,\sigma,s,\alpha,\beta}\left( \|(u,\theta)\|_{\dot{H}_{a,\sigma}^s}^{\frac{2\alpha}{2\alpha-1}} +
 \|(u,\theta)\|_{\dot{H}_{a,\sigma}^s}^{\frac{2\beta}{2\beta-1}}+1\right)\|(u,\theta)\|_{\dot{H}_{a,\sigma}^s}^2.
\end{align}

Now, let us assume that $\displaystyle \limsup_{t\nearrow T^*} \|(u,\theta)(t)\|_{\dot{H}_{a,\sigma}^{s}} < \infty$. Thus, apply Theorem \ref{teoremaexistenciaB} to obtain a positive constant $C_{a,\sigma,s}$ such that
\begin{align}\label{limitacao}
\displaystyle \|(u,\theta)(t)\|_{\dot{H}_{a,\sigma}^{s}} \leq C_{a,\sigma,s},\quad\forall\,t\in[0,T^*).
\end{align}
Thereby, by integrating over $[0,t]$   the inequality (\ref{esqueci}) and using (\ref{limitacao}), we can write the following inequality:
\begin{align*}
&\nonumber\|(u,\theta)(t)\|_{\dot{H}_{a,\sigma}^s}^2+\int_0^t\|u(\tau)\|_{\dot{H}_{a,\sigma}^{s+\alpha}}^2+\int_0^t\|\theta(\tau)\|_{\dot{H}_{a,\sigma}^{s+\beta}}^2\,d\tau \leq \|(u_0,\theta_0)\|_{\dot{H}_{a,\sigma}^s}^2+C_{a,\sigma,s,\alpha,\beta}T^*,
\end{align*}
for all $t\in[0,T^*).$ As a result, particularly, one has
\begin{align}\label{e4}
\int_0^t\|u(\tau)\|_{\dot{H}_{a,\sigma}^{s+\alpha}}^2d\tau+\int_0^t\|\theta(\tau)\|_{\dot{H}_{a,\sigma}^{s+\beta}}^2\,d\tau \leq C_{a,\sigma,s,\alpha,\beta,u_0,\theta_0,T^*},\quad\forall\,t\in[0,T^*).
\end{align}

Now, let  $(\kappa_n)_{n\in\mathbb{N}}$ be an arbitrary sequence
such that $0<\kappa_n<T^*$ and
$\kappa_n\nearrow T^*$. Hence, we claim that 
\begin{align}\label{e5}
\lim_{n,m\rightarrow \infty} \|(u,\theta)(\kappa_n)-(u,\theta)(\kappa_m)\|_{\dot{H}_{a,\sigma}^s}=0.
\end{align}
In fact, by the use of the equalities (\ref{wilber15}) and (\ref{wilber16}), one can rewrite the difference presented in (\ref{e5}) as follows:
\begin{align*}
(u,\theta)(\kappa_n)-(u,\theta)(\kappa_m)= I_1+I_2+I_3,
\end{align*}
where
\begin{align}\label{I1}
I_1&= (I_{11},I_{12}):= ([e^{-  \kappa_n (-\Delta)^\alpha}-e^{-\kappa_m(-\Delta)^\alpha}]u_0,[e^{-\kappa_n (-\Delta)^\beta}-e^{- \kappa_m (-\Delta)^\beta}]\theta_0),
\end{align}
and
\begin{align}\label{wilber19}
\nonumber &I_2=(I_{21}^{(1)}+I_{21}^{(2)},I_{22}):= \Big(\int_{0}^{\kappa_m}[e^{- (\kappa_m-\tau)(-\Delta)^\alpha}-e^{- (\kappa_n-\tau)(-\Delta)^\alpha}] \mathbb{P}[u\cdot \nabla u]\,d\tau+\\ &
\int_{0}^{\kappa_m}[e^{-  (\kappa_n-\tau)(-\Delta)^\alpha}-e^{- (\kappa_m-\tau)(-\Delta)^\alpha}] \mathbb{P}[\theta e_3]\,d\tau,\int_{0}^{\kappa_m}[e^{- (\kappa_m-\tau)(-\Delta)^\beta}-e^{- (\kappa_n-\tau)(-\Delta)^\beta}] (u\cdot \nabla \theta)\,d\tau\Big),
\end{align}
and also
\begin{align}\label{wilber20}
\nonumber I_3=(I_{31}^{(1)}+I_{31}^{(2)},I_{32})&:= \Big(\int_{\kappa_n}^{\kappa_m} e^{- (\kappa_n-\tau)(-\Delta)^\alpha} \mathbb{P}[u\cdot \nabla u]\,d\tau+
\int_{\kappa_m}^{\kappa_n} e^{- (\kappa_n-\tau)(-\Delta)^\alpha} \mathbb{P}[\theta e_3]\,d\tau,\\
&\quad\int_{\kappa_n}^{\kappa_m}e^{- (\kappa_n-\tau)(-\Delta)^\beta} (u\cdot \nabla \theta)\,d\tau\Big).
\end{align}
In order to estimate $I_1$ (see (\ref{I1})), we establish the following inequality:
\begin{align*}
\|I_{12}\|_{\dot{H}_{a,\sigma}^s}^2
&\leq \int_{\mathbb{R}^3} [e^{- \kappa_n |\xi|^{2\beta}}-e^{-T^* |\xi|^{2\beta}}]^2|\xi|^{2s} e^{2a|\xi|^{\frac{1}{\sigma}}} |\widehat{\theta}_0(\xi)|^2\,d\xi.
\end{align*}
Therefore, by applying Dominated Convergence Theorem, one concludes $\displaystyle\lim_{n,m\rightarrow \infty}\|I_{12}\|_{\dot{H}_{a,\sigma}^s}$ $=0$ (it is enough to recall that $\theta_0\in \dot{H}_{a,\sigma}^s$). By the use of analogous arguments, one obtains
$\displaystyle\lim_{n,m\rightarrow \infty}\|I_{11}\|_{\dot{H}_{a,\sigma}^s}=0.$ Thereby, we can assure that  $\displaystyle\lim_{n,m\rightarrow\infty}\|I_1\|_{\dot{H}_{a,\sigma}^s}=0$. 

 Cauchy-Schwarz's inequality implies that the next inequality holds:
\begin{align*}
\nonumber\|I_{21}^{(1)}\|_{\dot{H}_{a,\sigma}^s}&\leq \int_0^{\kappa_m}\Big(\int_{\mathbb{R}^3}[e^{-(\kappa_m-\tau)|\xi|^{2\alpha}}-e^{-(\kappa_n-\tau)|\xi|^{2\alpha}}]^2|\xi|^{2s} e^{2a|\xi|^{\frac{1}{\sigma}}}|\mathcal{F}[u\cdot\nabla u](\xi)|^2d\xi\Big)^{\frac{1}{2}} d\tau\\
&\leq\sqrt{T^*}\Big(\int_0^{T^*}\int_{\mathbb{R}^3}[1-e^{- (T^*-\kappa_m)|\xi|^{2\alpha}}]^2|\xi|^{2s} e^{2a|\xi|^{\frac{1}{\sigma}}}|\mathcal{F}[u\cdot\nabla u](\xi)|^2d\xi d\tau\Big)^{\frac{1}{2}}.
\end{align*}
Apply Lemma \ref{lemanovow2} ii) (recall that $a>0$, $\sigma>1$, $s\in[0,\frac{3}{2})$ and $\alpha\geq1$), (\ref{limitacao}), Hölder's inequality and (\ref{e4}) to infer
\begin{align*}
\nonumber\int_0^{T^*}\|u\cdot\nabla u\|_{\dot{H}_{a,\sigma}^s}^2 d\tau&\leq \int_0^{T^*}\|u\otimes u\|_{\dot{H}_{a,\sigma}^{s+1}}^2d\tau\leq C_{a,\sigma,s,\alpha} \int_0^{T^*}\| u\|_{\dot{H}_{a,\sigma}^{s+\alpha}}^{\frac{2}{\alpha}}d\tau\\
&\leq C_{a,\sigma,s,\alpha} \Big(\int_0^{T^*}\| u\|_{\dot{H}_{a,\sigma}^{s+\alpha}}^2d\tau\Big)^{\frac{1}{\alpha}}(T^*)^{1-\frac{1}{\alpha}}\leq C_{a,\sigma,s,\alpha,\beta,u_0,b_0,T^*}.
\end{align*}
Thereby,  Dominated Convergence Theorem implies the  limit
$\displaystyle \lim_{n,m\rightarrow\infty}\|I_{21}^{(1)}\|_{\dot{H}_{a,\sigma}^s}$ $=0.$ By the same arguments, one concludes $\displaystyle \lim_{n,m\rightarrow\infty}\|I_{22}\|_{\dot{H}_{a,\sigma}^s}=0$ (recall that $\beta\geq1$). On the other hand, we can write the estimate below:
\begin{align}\label{wilber21}
\nonumber\|I_{21}^{(2)}\|_{\dot{H}_{a,\sigma}^s}&\leq \int_0^{\kappa_m}\Big(\int_{\mathbb{R}^3}[e^{-(\kappa_n-\tau)|\xi|^{2\alpha}}-e^{-(\kappa_m-\tau)|\xi|^{2\alpha}}]^2|\xi|^{2s} e^{2a|\xi|^{\frac{1}{\sigma}}}|\widehat{\theta}(\xi)|^2d\xi\Big)^{\frac{1}{2}} d\tau\\
&\leq\sqrt{T^*}\Big(\int_0^{T^*}\int_{\mathbb{R}^3}[1-e^{- (T^*-\kappa_n)|\xi|^{2\alpha}}]^2|\xi|^{2s} e^{2a|\xi|^{\frac{1}{\sigma}}}|\widehat{\theta}(\xi)|^2d\xi d\tau\Big)^{\frac{1}{2}}.
\end{align}
On the other hand, by (\ref{limitacao}), it is true that
\begin{align}\label{wilber22}
\int_0^{T^*}\|\theta\|_{\dot{H}_{a,\sigma}^s}^2 d\tau\leq C_{a,\sigma,s}^2 T^*.
\end{align}
Consequently, by using Dominated Convergence Theorem again, we deduce
$\displaystyle \lim_{n,m\rightarrow\infty}\|I_{21}^{(2)}\|_{\dot{H}_{a,\sigma}^s}$ $=0.$ These results above imply that $\displaystyle\lim_{n,m\rightarrow\infty} \|I_2\|_{\dot{H}_{a,\sigma}^s}=0$. 

Lemma \ref{lemanovow2} ii), (\ref{limitacao}), Hölder's inequality and (\ref{e4}) imply the next inequalities: 
\begin{align*}
\|I_{31}^{(1)}\|_{\dot{H}_{a,\sigma}^s}&\leq\int_{\kappa_n}^{\kappa_m}\|e^{-(\kappa_n-\tau)(- \Delta)^\alpha}\mathbb{P}(u\cdot\nabla u)\|_{\dot{H}_{a,\sigma}^s}\,d\tau\\
\nonumber&\leq C_{a,\sigma,s,\alpha} \Big(\int_0^{T^*}\| u\|_{\dot{H}_{a,\sigma}^{s+\alpha}}^2d\tau\Big)^{\frac{1}{2\alpha}}(T^*-\kappa_n)^{1-\frac{1}{2\alpha}}\\
\nonumber&\leq C_{a,\sigma,s,\alpha,\beta,u_0,b_0,T^*}(T^*-\kappa_n)^{1-\frac{1}{2\alpha}}.
\end{align*}
As a result, it follows that $\displaystyle\lim_{n,m\rightarrow\infty} \|I_{31}^{(1)}\|_{\dot{H}_{a,\sigma}^s}=0$ (since $\kappa_n\nearrow T^*$ and $\alpha\geq 1$). Analogously, we infer $\displaystyle\lim_{n,m\rightarrow\infty} \|I_{32}\|_{\dot{H}_{a,\sigma}^s}=0$ (recall that $\beta\geq1$). Furthermore, by (\ref{limitacao}) again, we can write
\begin{align}\label{wilber23}
\nonumber\|I_{31}^{(2)}\|_{\dot{H}_{a,\sigma}^s}&\leq\int_{\kappa_m}^{\kappa_n}\|e^{-(\kappa_n-\tau)(- \Delta)^\alpha}\mathbb{P}[\theta e_3]\|_{\dot{H}_{a,\sigma}^s}\,d\tau\\
&\leq \int_{\kappa_m}^{T^*}\|\theta\|_{\dot{H}_{a,\sigma}^s}\,d\tau\leq C_{a,\sigma,s}(T^*-\kappa_m).
\end{align}
As a consequence, we obtain $\displaystyle\lim_{n,m\rightarrow\infty} \|I_{31}^{(2)}\|_{\dot{H}_{a,\sigma}^s}=0$ (since $\kappa_m\nearrow T^*$). Therefore, $\displaystyle\lim_{n,m\rightarrow\infty} \|I_{3}\|_{\dot{H}_{a,\sigma}^s}$ $=0$.

These arguments above show the veracity of  (\ref{e5}). This means that  $((u,\theta)(\kappa_n))_{n\in\mathbb{N}}$ is a Cauchy sequence in Banach space $\dot{H}_{a,\sigma}^s$ (since $s<\frac{3}{2}$).
Hence, the limit below holds:
\begin{align}\label{wilber2}
\lim_{n\rightarrow \infty} \|(u,\theta)(\kappa_n)- (u_1,\theta_1)\|_{\dot{H}_{a,\sigma}^s}=0,
\end{align}
for some $(u_1,\theta_1)\in \dot{H}_{a,\sigma}^{s}$.

Now, let us prove that  $(u_1, \theta_1)$ does not rely on 
$(\kappa_n)_{n\in\mathbb{N}}$. In fact, assume that we have another sequence $(\rho_n)_{n\in\mathbb{N}}\subseteq (0,T^*)$ and $(u_2,\theta_2)\in \dot{H}_{a,\sigma}^s$  such that $\rho_n \nearrow T^*$ and
$$\displaystyle\lim_{n\rightarrow \infty} \|(u,\theta)(\rho_n)- (u_2,\theta_2)\|_{\dot{H}_{a,\sigma}^s}=0.$$
Allow us to show that  $(u_2,\theta_2)=(u_1,\theta_1).$ First of all, consider that $(\varsigma_n)_{n\in\mathbb{N}}\subseteq (0,T^*)$ is given by 
$\varsigma_{2n}=\kappa_n$ and $\varsigma_{2n-1}=\rho_n$, for all $n\in\mathbb{N}$. Thus, it is easy to notice that $\varsigma_n\nearrow T^*$
and that there is $(u_3,\theta_3)\in \dot{H}_{a,\sigma}^s$ that satisfies 
$$\displaystyle\lim_{n\rightarrow \infty} \|(u,\theta)(\varsigma_n)- (u_3,\theta_3)\|_{\dot{H}_{a,\sigma}^s}=0,$$
as we have done previously. Thereby, we conclude $(u_1,\theta_1)=(u_3,\theta_3)=(u_2,\theta_2).$ 

At last, (\ref{wilber2}) leads us to 
\begin{align}\label{wilber3}
\displaystyle\lim_{t\nearrow T^*} \|(u,\theta)(t)- (u_1,\theta_1)\|_{\dot{H}_{a,\sigma}^s}=0.
\end{align}
In order to finish this proof, assume that $(\widetilde{u},\widetilde{\theta})\in C_{T}(\dot{H}_{a,\sigma}^s)$ is the solution for the Boussinesq equations (\ref{micropolar}) with the initial data $(u_1,\theta_1)$
(it is enough to apply Theorem \ref{teoremaexistenciaB}). Hence, we can define
$$(\overline{u},\overline{\theta})(t)=\left\{
                                     \begin{array}{ll}
                                       (u,\theta)(t), & t\in [0,T^*); \\
                                       (\widetilde{u},\widetilde{\theta})(t-T^*), & t\in [T^*,T+T^*].
                                     \end{array}
                                   \right.
$$
As a result, by (\ref{wilber3}),  $(\overline{u},\overline{\theta})\in C_{T+T^*}(\dot{H}_{a,\sigma}^s)$ is a solution for the Boussinesq equations (\ref{micropolar}) that is defined beyond $T^*$ with initial data $(u_0,\theta_0)$. This is not possible because of the maximality of $T^*$! Consequently, we can write
\begin{align}\label{wilber5}
\displaystyle \limsup_{t\nearrow T^*} \|(u,\theta)(t)\|_{\dot{H}_{a,\sigma}^{s}} = \infty.
\end{align}
That is, the proof of Theorem   \ref{teoremaB} \textbf{i)} with $n=1$ is established.

\caixa

\hspace{-0.5cm}\textbf{Proof of Theorem \ref{teoremaB} \textbf{ii)} with $n=1$:}
\vspace{0.2cm}

At first, apply Hölder's inequality  to the inequality (\ref{n3}) in order to reach the inequality below:
\begin{align}\label{wilber4}
\nonumber\frac{d}{dt}\|(u,\theta)(t)\|_{\dot{H}_{a,\sigma}^s}^2+\|u(t)\|_{\dot{H}_{a,\sigma}^{s+\alpha}}^2+\|\theta(t)\|_{\dot{H}_{a,\sigma}^{s+\beta}}^2 &\leq C_{s,\alpha,\beta} [\|e^{\frac{a}{\sigma}|\cdot|^{\frac{1}{\sigma}}}(\widehat{u},\widehat{\theta})\|_{L^{1}}^{\frac{2\alpha}{2\alpha-1}}+
\|e^{\frac{a}{\sigma}|\cdot|^{\frac{1}{\sigma}}}(\widehat{u},\widehat{\theta})\|_{L^{1}}^{\frac{2\beta}{2\beta-1}}+1]\\
&\quad\times\|(u,\theta)\|_{\dot{H}_{a,\sigma}^{s}}^{2}.
\end{align}
Secondly, by applying  Gronwall's inequality to (\ref{wilber4}), we can write the following inequality:
\begin{align}\label{wilber24}
\nonumber\|(u,b)(T)\|_{\dot{H}_{a,\sigma}^s}^2&\leq \|(u,b)(t)\|_{\dot{H}_{a,\sigma}^s}^2\exp \{C_{s,\alpha,\beta}\int_t^T [\|e^{\frac{a}{\sigma}|\cdot|^{\frac{1}{\sigma}}}(\widehat{u},\widehat{\theta})\|_{L^{1}}^{\frac{2\alpha}{2\alpha-1}}+
\|e^{\frac{a}{\sigma}|\cdot|^{\frac{1}{\sigma}}}(\widehat{u},\widehat{\theta})\|_{L^{1}}^{\frac{2\beta}{2\beta-1}}+1]\,d\tau \}\\
&\leq \|(u,b)(t)\|_{\dot{H}_{a,\sigma}^s}^2e^{C_{s,\alpha,\beta}(T-t)}\exp \{C_{s,\alpha,\beta}\int_t^T [\|e^{\frac{a}{\sigma}|\cdot|^{\frac{1}{\sigma}}}(\widehat{u},\widehat{\theta})\|_{L^{1}}^{\frac{2\alpha}{2\alpha-1}}+
\|e^{\frac{a}{\sigma}|\cdot|^{\frac{1}{\sigma}}}(\widehat{u},\widehat{\theta})\|_{L^{1}}^{\frac{2\beta}{2\beta-1}}]\,d\tau \},
\end{align}
for all $0\leq t\leq T< T^*$. At last, by taking the limit superior, as $T\nearrow T^*$, the equality (\ref{wilber5}) leads us to
\begin{align}\label{wilber8}
\int_t^{T^*} [\|e^{\frac{a}{\sigma}|\cdot|^{\frac{1}{\sigma}}}(\widehat{u},\widehat{\theta})(\tau)\|_{L^{1}}^{\frac{2\alpha}{2\alpha-1}}+
\|e^{\frac{a}{\sigma}|\cdot|^{\frac{1}{\sigma}}}(\widehat{u},\widehat{\theta})(\tau)\|_{L^{1}}^{\frac{2\beta}{2\beta-1}}]\,d\tau=\infty,
\end{align}
for all $t\in[0,T^*)$. Therefore, the  proof of Theorem  \ref{teoremaB} \textbf{ii)} with $n=1$ is complete.

\caixa

\hspace{-0.5cm}\textbf{Proof of Theorem  \ref{teoremaB} \textbf{iii)} with $n=1$:}\label{secaoesqueci}
\vspace{0.2cm}

By taking Fourier transform in the first equation
of (\ref{micropolar}) and the scalar product  with $\widehat{u}(t)$ of the result obtained, it follows that
\begin{align}\label{wilber25}
 \frac{1}{2}\partial_t|\widehat{u}(t)|^2+|\xi|^{2\alpha}|\widehat{ u}|^2 &\leq
|\widehat{u}\cdot \widehat{u\cdot\nabla \displaystyle u}|+|\widehat{u}_3 \widehat{\theta}|.
\end{align}
%Similarly, we have
%\begin{align}\label{10}
% \frac{1}{2}\partial_t|\widehat{b}(t)|^2+\nu|\xi|^{2\beta}|\widehat{ b}|^2 &\leq
%|\widehat{b}\cdot \widehat{u\cdot\nabla \displaystyle b}|+|\widehat{b}\cdot \widehat{b\cdot\nabla u}|.
%\end{align}
Now, let $\delta$ be any positive real number to infer
\begin{align*}
\nonumber \partial_{t}\sqrt{|\widehat{u}(t)|^2+\delta}+\frac{|\xi|^{2\alpha}|\widehat{u}|^2}{\sqrt{|\widehat{u}|^2+\delta}}
&\leq|\widehat{u\cdot\nabla \displaystyle u}|+|\widehat{\theta}|.\
\end{align*}
By integrating the inequality above, one can write the next result: 
\begin{align*}
\nonumber \sqrt{|\widehat{u}(T)|^2+\delta}+ |\xi|^{2\alpha}\int_{t}^{T}\frac{|\widehat{u}(\tau)|^2}{\sqrt{|\widehat{u}(\tau)|^2+\delta}}\,d\tau\leq\sqrt{|\widehat{u}(t)|^2+\delta}
+\int_{t}^{T}[|\widehat{(u\cdot\nabla \displaystyle u)}(\tau)|+|\widehat{\theta}(\tau)|]\,d\tau,
\end{align*}
 where  $0\leq t\leq T<T^{*}<\infty$. Thereby,  take  the limit in the inequality above, as $\delta\searrow 0$ and, after that, multiply by $e^{\frac{a}{\sigma}|\xi|^{\frac{1}{\sigma}}}$ and integrate over $\xi \in\mathbb{R}^3$ the  result obtained to deduce that
\begin{align}\label{wilber6}
 \nonumber\|e^{\frac{a}{\sigma}|\cdot|^{\frac{1}{\sigma}}}\widehat{u}(T)\|_{L^1}+ \int_{t}^{T}\|e^{\frac{a}{\sigma}|\cdot|^{\frac{1}{\sigma}}}\mathcal{F}[(-\Delta)^{\alpha} u](\tau)\|_{L ^1}\,d\tau&\leq \|e^{\frac{a}{\sigma}|\cdot|^{\frac{1}{\sigma}}}\widehat{u}(t)\|_{L^1}\\
&\quad+\int_{t}^{T}\int_{\mathbb{R}^3}e^{\frac{a}{\sigma}|\xi|^{\frac{1}{\sigma}}}[|\widehat{(u\cdot\nabla \displaystyle u)}(\tau)|+|\widehat{\theta}(\tau)|]\,d\xi d\tau.\
\end{align}
By observing the second equation in (\ref{micropolar}) and following analogous arguments as above, one has
\begin{align}\label{wilber7}
 &\|e^{\frac{a}{\sigma}|\cdot|^{\frac{1}{\sigma}}}\widehat{\theta}(T)\|_{L^1}+\int_{t}^{T}\|e^{\frac{a}{\sigma}|\cdot|^{\frac{1}{\sigma}}}\mathcal{F}[(-\Delta)^{\beta} \theta](\tau)\|_{L ^1}\,d\tau\leq \|e^{\frac{a}{\sigma}|\cdot|^{\frac{1}{\sigma}}}\widehat{\theta}(t)\|_{L^1}+\int_{t}^{T}\int_{\mathbb{R}^3}e^{\frac{a}{\sigma}|\xi|^{\frac{1}{\sigma}}}|\widehat{(u\cdot\nabla \displaystyle \theta)}(\tau)|\,d\xi d\tau.\
\end{align}
As a result, by (\ref{wilber6}) and (\ref{wilber7}), we can write
\begin{align}\label{wilber26}
 \nonumber&\|e^{\frac{a}{\sigma}|\cdot|^{\frac{1}{\sigma}}}(\widehat{u},\widehat{\theta})(T)\|_{L^1}+ \int_{t}^{T}\|e^{\frac{a}{\sigma}|\cdot|^{\frac{1}{\sigma}}}\mathcal{F}[(-\Delta)^{\alpha} u](\tau)\|_{L ^1}\,d\tau+ \int_{t}^{T}\|e^{\frac{a}{\sigma}|\cdot|^{\frac{1}{\sigma}}}\mathcal{F}[(-\Delta)^{\beta} \theta](\tau)\|_{L ^1}\,d\tau\\
\nonumber&\leq \|e^{\frac{a}{\sigma}|\cdot|^{\frac{1}{\sigma}}}(\widehat{u},\widehat{\theta})(t)\|_{L^1}
+\int_{t}^{T}\int_{\mathbb{R}^3}e^{\frac{a}{\sigma}|\xi|^{\frac{1}{\sigma}}}[|\widehat{(u\cdot\nabla \displaystyle u)}(\tau)|+|\widehat{\theta}(\tau)|+|\widehat{(u\cdot\nabla \displaystyle \theta)}(\tau)|]\,d\xi d\tau\\
&\leq  \|e^{\frac{a}{\sigma}|\cdot|^{\frac{1}{\sigma}}}(\widehat{u},\widehat{\theta})(t)\|_{L^1}
+\int_{t}^{T}\int_{\mathbb{R}^3}e^{\frac{a}{\sigma}|\xi|^{\frac{1}{\sigma}}}[|\widehat{(u\cdot\nabla \displaystyle u)}(\tau)|+|\widehat{(u\cdot\nabla \displaystyle \theta)}(\tau)|]\,d\xi d\tau+\int_{t}^{T}\|e^{\frac{a}{\sigma}|\cdot|^{\frac{1}{\sigma}}}\widehat{\theta}(\tau)\|_{L^1} d\tau.
\end{align}
Lemma \ref{benameur} (recall that $a>0$ and $\sigma>1$) and Young and Hölder's (recall that $\beta\geq1$) inequalities imply that
\begin{align*}
\nonumber\int_{\mathbb{R}^3}e^{\frac{a}{\sigma}|\xi|^{\frac{1}{\sigma}}}|\widehat{(u\cdot\nabla \displaystyle \theta)}(\xi)|\,d\xi&\leq
(2\pi)^{-3}\|[e^{\frac{a}{\sigma}|\cdot|^{\frac{1}{\sigma}}}|\widehat{u}|]\ast[e^{\frac{a}{\sigma}|\cdot|^{\frac{1}{\sigma}}}|\widehat{ \nabla \theta}|]\|_{L^1}\leq   (2\pi)^{-3}\|e^{\frac{a}{\sigma}|\cdot|^{\frac{1}{\sigma}}}\widehat{u}\|_{L^1} \|e^{\frac{a}{\sigma}|\cdot|^{\frac{1}{\sigma}}}\widehat{ \nabla \theta}\|_{L^1}\\
&\leq (2\pi)^{-3}\|e^{\frac{a}{\sigma}|\cdot|^{\frac{1}{\sigma}}}\widehat{u}\|_{L^1}\|e^{\frac{a}{\sigma}|\cdot|^{\frac{1}{\sigma}}}\widehat{ \theta}\|_{L^1}^{1-\frac{1}{2\beta}} \|e^{\frac{a}{\sigma}|\cdot|^{\frac{1}{\sigma}}}\mathcal{F}[(-\Delta)^{\beta} \theta]\|_{L^1}^{\frac{1}{2\beta}}.
\end{align*}
Similarly, one infers (recall that $a>0$, $\sigma>1$ and $\alpha\geq1$)
\begin{align*}
\nonumber\int_{\mathbb{R}^3}e^{\frac{a}{\sigma}|\xi|^{\frac{1}{\sigma}}}|\widehat{(u\cdot\nabla \displaystyle u)}(\xi)|\,d\xi\leq (2\pi)^{-3}\|e^{\frac{a}{\sigma}|\cdot|^{\frac{1}{\sigma}}}\widehat{u}\|_{L^1}\|e^{\frac{a}{\sigma}|\cdot|^{\frac{1}{\sigma}}}\widehat{ u}\|_{L^1}^{1-\frac{1}{2\alpha}} \|e^{\frac{a}{\sigma}|\cdot|^{\frac{1}{\sigma}}}\mathcal{F}[(-\Delta)^{\alpha} u]\|_{L^1}^{\frac{1}{2\alpha}}.
\end{align*}
Consequently, by Young's inequality, one deduces
\begin{align}\label{wilber27}
 \nonumber&\|e^{\frac{a}{\sigma}|\cdot|^{\frac{1}{\sigma}}}(\widehat{u},\widehat{\theta})(T)\|_{L^1}+\frac{1 }{2} \int_{t}^{T}\|e^{\frac{a}{\sigma}|\cdot|^{\frac{1}{\sigma}}}\mathcal{F}[(-\Delta)^{\alpha} u](\tau)\|_{L ^1}\,d\tau+\frac{1 }{2} \int_{t}^{T}\|e^{\frac{a}{\sigma}|\cdot|^{\frac{1}{\sigma}}}\mathcal{F}[(-\Delta)^{\beta} \theta](\tau)\|_{L ^1}\,d\tau\\
\nonumber &\leq \|e^{\frac{a}{\sigma}|\cdot|^{\frac{1}{\sigma}}}(\widehat{u},\widehat{\theta})(t)\|_{L^1}
+C_{\alpha,\beta}\int_{t}^{T}\|e^{\frac{a}{\sigma}|\cdot|^{\frac{1}{\sigma}}} (\widehat{u},\widehat{\theta})(\tau)\|_{L^1}
\{\|e^{\frac{a}{\sigma}|\cdot|^{\frac{1}{\sigma}}} (\widehat{u},\widehat{\theta})(\tau)\|_{L^1}^{\frac{2\alpha}{2\alpha-1}}+\|e^{\frac{a}{\sigma}|\cdot|^{\frac{1}{\sigma}}} (\widehat{u},\widehat{\theta})(\tau)\|_{L^1}^{\frac{2\beta}{2\beta-1}}\}\, d\tau\\
\nonumber&\quad+\int_{t}^{T}\|e^{\frac{a}{\sigma}|\cdot|^{\frac{1}{\sigma}}}\widehat{\theta}(\tau)\|_{L^1} d\tau\\
&\leq\|e^{\frac{a}{\sigma}|\cdot|^{\frac{1}{\sigma}}}(\widehat{u},\widehat{\theta})(t)\|_{L^1}
+C_{\alpha,\beta}\int_{t}^{T}\|e^{\frac{a}{\sigma}|\cdot|^{\frac{1}{\sigma}}} (\widehat{u},\widehat{\theta})(\tau)\|_{L^1}
\{\|e^{\frac{a}{\sigma}|\cdot|^{\frac{1}{\sigma}}} (\widehat{u},\widehat{\theta})(\tau)\|_{L^1}^{\frac{2\alpha}{2\alpha-1}}+\|e^{\frac{a}{\sigma}|\cdot|^{\frac{1}{\sigma}}} (\widehat{u},\widehat{\theta})(\tau)\|_{L^1}^{\frac{2\beta}{2\beta-1}}+1\}\, d\tau.
\end{align}
Now, we are able to apply Gronwall's inequality (with $0\leq t\leq T< T^*<\infty$ as previously) in order to obtain
\begin{align}\label{wilber28}
 \nonumber&\|e^{\frac{a}{\sigma}|\cdot|^{\frac{1}{\sigma}}}(\widehat{u},\widehat{\theta})(T)\|_{L^1}^{\frac{2\alpha}{2\alpha-1}}+
\|e^{\frac{a}{\sigma}|\cdot|^{\frac{1}{\sigma}}}(\widehat{u},\widehat{\theta})(T)\|_{L^1}^{\frac{2\beta}{2\beta-1}}\leq \{\|e^{\frac{a}{\sigma}|\cdot|^{\frac{1}{\sigma}}}(\widehat{u},\widehat{\theta})(t)\|_{L^1}^{\frac{2\alpha}{2\alpha-1}}+
\|e^{\frac{a}{\sigma}|\cdot|^{\frac{1}{\sigma}}}(\widehat{u},\widehat{\theta})(t)\|_{L^1}^{\frac{2\beta}{2\beta-1}}\}\\
\nonumber&\quad\times\exp\{C_{\alpha,\beta}\int_{t}^{T}\{\|e^{\frac{a}{\sigma}|\cdot|^{\frac{1}{\sigma}}} (\widehat{u},\widehat{\theta})(\tau)\|_{L^1}^{\frac{2\alpha}{2\alpha-1}}+\|e^{\frac{a}{\sigma}|\cdot|^{\frac{1}{\sigma}}} (\widehat{u},\widehat{\theta})(\tau)\|_{L^1}^{\frac{2\beta}{2\beta-1}}+1\}\, d\tau\}\\
\nonumber&\leq \{\|e^{\frac{a}{\sigma}|\cdot|^{\frac{1}{\sigma}}}(\widehat{u},\widehat{\theta})(t)\|_{L^1}^{\frac{2\alpha}{2\alpha-1}}+
\|e^{\frac{a}{\sigma}|\cdot|^{\frac{1}{\sigma}}}(\widehat{u},\widehat{\theta})(t)\|_{L^1}^{\frac{2\beta}{2\beta-1}}\}\\
&\quad\times e^{C_{\alpha,\beta}(T-t)}\exp\{C_{\alpha,\beta}\int_{t}^{T}\{\|e^{\frac{a}{\sigma}|\cdot|^{\frac{1}{\sigma}}} (\widehat{u},\widehat{\theta})(\tau)\|_{L^1}^{\frac{2\alpha}{2\alpha-1}}+\|e^{\frac{a}{\sigma}|\cdot|^{\frac{1}{\sigma}}} (\widehat{u},\widehat{\theta})(\tau)\|_{L^1}^{\frac{2\beta}{2\beta-1}}\}\, d\tau\}.
\end{align}
Hence, we can write the following inequality:
\begin{align}\label{wilber29}
 \nonumber&-C_{\alpha,\beta}^{-1}\frac{d}{dT}[\exp\{-C_{\alpha,\beta}\int_{t}^{T}\{\|e^{\frac{a}{\sigma}|\cdot|^{\frac{1}{\sigma}}} (\widehat{u},\widehat{\theta})(\tau)\|_{L^1}^{\frac{2\alpha}{2\alpha-1}}+\|e^{\frac{a}{\sigma}|\cdot|^{\frac{1}{\sigma}}} (\widehat{u},\widehat{\theta})(\tau)\|_{L^1}^{\frac{2\beta}{2\beta-1}}\}\,d\tau\}]\\
&\leq e^{C_{\alpha,\beta}(T-t)}\{\|e^{\frac{a}{\sigma}|\cdot|^{\frac{1}{\sigma}}}(\widehat{u},\widehat{\theta})(t)\|_{L^1}^{\frac{2\alpha}{2\alpha-1}}+
\|e^{\frac{a}{\sigma}|\cdot|^{\frac{1}{\sigma}}}(\widehat{u},\widehat{\theta})(t)\|_{L^1}^{\frac{2\beta}{2\beta-1}}\}. 
\end{align}
By integrating over $[t,t_0]$ the inequality above, one concludes
\begin{align}\label{wilber30}
 \nonumber&-C_{\alpha,\beta}^{-1}\exp\left(-C_{\alpha,\beta}\int_{t}^{t_0}\left(\|e^{\frac{a}{\sigma}|\cdot|^{\frac{1}{\sigma}}} (\widehat{u},\widehat{\theta})(\tau)\|_{L^1}^{\frac{2\alpha}{2\alpha-1}}+\|e^{\frac{a}{\sigma}|\cdot|^{\frac{1}{\sigma}}} (\widehat{u},\widehat{\theta})(\tau)\|_{L^1}^{\frac{2\beta}{2\beta-1}}\right)\,d\tau\right)+C_{\alpha,\beta}^{-1}\\
&\leq [C_{\alpha,\beta}^{-1}e^{C_{\alpha,\beta}(t_0-t)}-C_{\alpha,\beta}^{-1}]\left(\|e^{\frac{a}{\sigma}|\cdot|^{\frac{1}{\sigma}}}(\widehat{u},\widehat{\theta})(t)\|_{L^1}^{\frac{2\alpha}{2\alpha-1}}+
\|e^{\frac{a}{\sigma}|\cdot|^{\frac{1}{\sigma}}}(\widehat{u},\widehat{\theta})(t)\|_{L^1}^{\frac{2\beta}{2\beta-1}}\right). 
\end{align}
where $0\leq t\leq t_0< T^*$. Thus, we are ready to take the limit in the inequality above, as $t_0 \nearrow T^*$, and apply (\ref{wilber8}) in order to reach
\begin{align}\label{wilber31}
C_{\alpha,\beta}^{-1}\leq   [C_{\alpha,\beta}^{-1}e^{C_{\alpha,\beta}(T^*-t)}-C_{\alpha,\beta}^{-1}]\{\|e^{\frac{a}{\sigma}|\cdot|^{\frac{1}{\sigma}}}(\widehat{u},\widehat{\theta})(t)\|_{L^1}^{\frac{2\alpha}{2\alpha-1}}+
\|e^{\frac{a}{\sigma}|\cdot|^{\frac{1}{\sigma}}}(\widehat{u},\widehat{\theta})(t)\|_{L^1}^{\frac{2\beta}{2\beta-1}}\},
\end{align}
for all $t\in [0,T^*)$. 
As a result, one infers
\begin{align}\label{wilber9}
\|e^{\frac{a}{\sigma}|\cdot|^{\frac{1}{\sigma}}}(\widehat{u},\widehat{\theta})(t)\|_{L^1}^{\frac{2\alpha}{2\alpha-1}}+
\|e^{\frac{a}{\sigma}|\cdot|^{\frac{1}{\sigma}}}(\widehat{u},\widehat{\theta})(t)\|_{L^1}^{\frac{2\beta}{2\beta-1}}\geq  [e^{C_{\alpha,\beta}(T^*-t)}-1]^{-1},
\end{align}
for all $t\in [0,T^*)$. Therefore, the proof of Theorem \ref{teoremaB} \textbf{iii)}  for $n=1$ is established.

\caixa

\hspace{-0.5cm}\textbf{Proof of Theorem \ref{teoremaB} \textbf{iv)} with $n=1$:}\label{secaoteoremaexistenciaB}
\vspace{0.2cm}

First of all, notice that  $\dot{H}_{a,\sigma}^s\hookrightarrow \dot{H}_{\frac{a}{\sqrt{\sigma}},\sigma}^s$
(since $\sigma>  1$). Actually, more precisely, the inequality $\displaystyle\|\cdot\|_{\dot{H}_{\frac{a}{\sqrt{\sigma}},\sigma}^s}\leq \|\cdot\|_{\dot{H}_{a,\sigma}^s}$ holds. As a result, we have$\footnote{From now on $T^*_{\omega}< \infty$ denotes the first blow-up time for the solution $(u,\theta)\in C([0,T^*_{\omega}); \dot{H}_{\omega,\sigma}^{s})$, where $\omega>0.$}$ $(u,\theta)\in C([0,T_{a}^*),$ $\dot{H}_{\frac{a}{\sqrt{\sigma}},\sigma}^s)$ (since $(u,\theta)\in C([0,T_{a}^*),$ $\dot{H}_{a,\sigma}^s)$) and, as a direct consequence, 
\begin{align}\label{wilber10}
T_{\frac{a}{\sqrt{\sigma}}}^*\geq T_a^*.
\end{align}
Furthermore, we can apply (\ref{wilber9}) and Cauchy-Schwarz's inequality  to conclude
\begin{align*}
\nonumber [e^{C_{\alpha,\beta}(T^{*}_a-t)}-1]^{-1}&\leq  \|e^{\frac{a}{\sigma}|\cdot|^{\frac{1}{\sigma}}} (\widehat{u},\widehat{\theta})(t)\|_{L^1}^{\frac{2\alpha}{2\alpha-1}}+\|e^{\frac{a}{\sigma}|\cdot|^{\frac{1}{\sigma}}} (\widehat{u},\widehat{\theta})(t)\|_{L^1}^{\frac{2\beta}{2\beta-1}}\\
&\leq C_{a,\sigma,s,\alpha,\beta}[\|(u,\theta)(t)\|_{\dot{H}_{\frac{a}{\sqrt{\sigma}},\sigma}^s}^{\frac{2\alpha}{2\alpha-1}}+\|(u,\theta)(t)\|_{\dot{H}_{\frac{a}{\sqrt{\sigma}},\sigma}^s}^{\frac{2\beta}{2\beta-1}}],
\end{align*}
for all $t\in [0,T_a^*)$ (recall that $\sigma>1$, $a>0$ and $s<\frac{3}{2}$). This is equivalent to
\begin{align}\label{i)n=1}
 \|(u,\theta)(t)\|_{\dot{H}_{\frac{a}{\sqrt{\sigma}},\sigma}^s}^{\frac{2\alpha}{2\alpha-1}}
+\|(u,\theta)(t)\|_{\dot{H}_{\frac{a}{\sqrt{\sigma}},\sigma}^s}^{\frac{2\beta}{2\beta-1}}\geq C_{a,\sigma,s,\alpha,\beta}[e^{C_{\alpha,\beta}(T^{*}_a-t)}-1]^{-1},
\end{align}
for all $t\in [0,T_a^*)$. Therefore, the proof of Theorem  \ref{teoremaB} \textbf{iv)} with $n=1$ is given.

\caixa

\hspace{-0.5cm}\textbf{Proof of Theorem  \ref{teoremaB}  with $n>1$:}\label{secaoteoremaexistenciaBnew}
\vspace{0.2cm}

Notice that, it follows directly from (\ref{i)n=1})  that
\begin{align*}
\limsup_{t\nearrow T_a^*}[\|(u,\theta)(t)\|_{\dot{H}_{\frac{a}{\sqrt{\sigma}},\sigma}^s}^{\frac{2\alpha}{2\alpha-1}}+\|(u,\theta)(t)\|_{\dot{H}_{\frac{a}{\sqrt{\sigma}},\sigma}^s}^{\frac{2\beta}{2\beta-1}}]=\infty,
\end{align*}
which implies, by a contradiction argument, the following limit:
\begin{align}\label{esqueci2}
\limsup_{t\nearrow T_a^*}\|(u,\theta)(t)\|_{\dot{H}_{\frac{a}{\sqrt{\sigma}},\sigma}^s}=\infty.
\end{align}
This result establishes the proof of Theorem \ref{teoremaB} \textbf{i)} with $n=2$. 

Therefore, we can restart and follow the steps of the proof of (\ref{wilber8}) and (\ref{wilber9}) and infer
$$\int_t^{T_a^*} [\|e^{\frac{a}{\sigma\sqrt{\sigma}}|\cdot|^{\frac{1}{\sigma}}} (\widehat{u},\widehat{\theta})(\tau)\|_{L^1}^{\frac{2\alpha}{2\alpha-1}}+\|e^{\frac{a}{\sigma\sqrt{\sigma}}|\cdot|^{\frac{1}{\sigma}}} (\widehat{u},\widehat{\theta})(\tau)\|_{L^1}^{\frac{2\beta}{2\beta-1}}]\,d\tau=\infty,\quad \forall t\in[0,T_a^*),$$
and also
\begin{align*}
\|e^{\frac{a}{\sigma\sqrt{\sigma}}|\cdot|^{\frac{1}{\sigma}}}(\widehat{u},\widehat{\theta})(t)\|_{L^1}^{\frac{2\alpha}{2\alpha-1}}+
\|e^{\frac{a}{\sigma\sqrt{\sigma}}|\cdot|^{\frac{1}{\sigma}}}(\widehat{u},\widehat{\theta})(t)\|_{L^1}^{\frac{2\beta}{2\beta-1}}\geq  [e^{C_{\alpha,\beta}(T^*_a-t)}-1]^{-1},\quad \forall t\in [0,T^*_a).
\end{align*}
 Consequently, the proof of Theorem  \ref{teoremaB} \textbf{ii)} and \textbf{iii)} with $n=2$ are given. In addition,  (\ref{esqueci2}) also implies that $T_a^*\geq T_{\frac{a}{\sqrt{\sigma}}}^*.$ Hence, by (\ref{wilber10}), one concludes 
\begin{align}\label{wilber11}
T_a^*=T_{\frac{a}{\sqrt{\sigma}}}^*.
\end{align}

Now, by (\ref{i)n=1}) (with $\frac{a}{\sqrt{\sigma}}$ replacing $a$),  one deduces
\begin{align*}
 \|(u,\theta)(t)\|_{\dot{H}_{\frac{a}{\sigma},\sigma}^s}^{\frac{2\alpha}{2\alpha-1}}
+\|(u,\theta)(t)\|_{\dot{H}_{\frac{a}{\sigma},\sigma}^s}^{\frac{2\beta}{2\beta-1}}\geq C_{a,\sigma,s,\alpha,\beta}[e^{C_{\alpha,\beta}(T^{*}_{\frac{a}{\sqrt{\sigma}}}-t)}-1]^{-1},
\end{align*}
for all $t\in [0,T_{\frac{a}{\sqrt{\sigma}}}^*)$ and, by (\ref{wilber11}), it follows that
\begin{align}\label{i)n=2}
 \|(u,\theta)(t)\|_{\dot{H}_{\frac{a}{\sigma},\sigma}^s}^{\frac{2\alpha}{2\alpha-1}}
+\|(u,\theta)(t)\|_{\dot{H}_{\frac{a}{\sigma},\sigma}^s}^{\frac{2\beta}{2\beta-1}}\geq C_{a,\sigma,s,\alpha,\beta}[e^{C_{\alpha,\beta}(T^{*}_{a}-t)}-1]^{-1},
\end{align}
for all $t\in [0,T_{a}^*)$. Therefore, the inequality (\ref{i)n=2}) proves Theorem  \ref{teoremaB} \textbf{iv)} with $n=2$. This means that Theorem  \ref{teoremaB} is proved for $n=2$.

Lastly, it is enough take the limit in the inequality (\ref{i)n=2}), as   $t\nearrow T^*_a$, to conclude that
$$\displaystyle \limsup_{t\nearrow T_a^*}\|(u,\theta)(t)\|_{\dot{H}_{\frac{a}{\sigma},\sigma}^s(\mathbb{R}^3)}=\infty$$
and, as a result, we are able to reestablish the process above and obtain, inductively, the proof of Theorem \ref{teoremaB} for any $n\in \mathbb{N}$.

\caixa

\subsection{Proof of Corollary \ref{corolario}}

%\hspace{-0.5cm}\textbf{Proof of Corollary \ref{corolario}:}\label{secaoteoremaexistenciaB1}
%\vspace{0.2cm}

Let us apply Dominated Convergence Theorem to Theorem \ref{teoremaB} iii) (recall that $\sigma>1$), with $\alpha=\beta$, in order to deduce
\begin{align}\label{corolarioH}
\|(\widehat{u},\widehat{\theta})(t)\|_{L^1}\geq  C'_{\alpha}[e^{C_{\alpha}(T^*-t)}-1]^{-\frac{2\alpha-1}{2\alpha}},\quad\forall t\in [0,T^*).
\end{align}

By the use of Lemma \ref{benameur1} (recall that $s\geq0$), with $k\in \mathbb{N}$ and $k\geq 2\sigma\mu$ (where $\mu>\frac{3}{2}$ is a constant), Lemma \ref{lemalorenz1} (recall that $a>0$, $\sigma>1$ and $s\in[0,\frac{3}{2})$)   and the inequality (\ref{corolarioH}), one has
\begin{align}\label{wilber12}
C'_{\alpha}[e^{C_{\alpha}(T^*-t)}-1]^{-\frac{2\alpha-1}{2\alpha}}\leq   C_s \|(u,\theta)(t)\|_{L^2}^{1-\frac{3}{2(s+\frac{k}{2\sigma})}}\|(u,\theta)(t)\|_{\dot{H}^{s+\frac{k}{2\sigma}} }^{\frac{3}{2(s+\frac{k}{2\sigma})}}, \quad \forall t\in[0,T^*).
\end{align}

On the other hand, take  $L^2$-inner product of the Boussinesq equations (\ref{micropolar}), with $u$ and $\theta$, respectively, and integrate the results obtained over $[0,t]$ in order to be capable of writing the following inequality:
\begin{align}\label{wilber32}
\frac{1}{2}\frac{d}{dt}\|(u,\theta)(t)\|_{L^2}^2+\|(-\Delta)^{\frac{\alpha}{2}}u(t)\|_{L^2}^2+\|(-\Delta)^{\frac{\beta}{2}}\theta(t)\|_{L^2}^2 &\leq \|(u,\theta)(t)\|_{L^2}^{2},\quad\forall t\in [0,T^*).
\end{align}
By Gronwall's Lemma, we have 
\begin{align}\label{normal2}
 \|(u,\theta)(t)\|_{L^2}\leq e^{T^*}\|(u_0,\theta_0)\|_{L^2},\quad\forall t\in [0,T^*).
\end{align}
Thereby, by using (\ref{wilber12}) and (\ref{normal2}), one infers
\begin{align*}
\frac{C_{s,\alpha,u_0,\theta_0,T^*}}{[e^{C_\alpha(T^*-t)}-1]^{\frac{2s(2\alpha-1)}{3\alpha}}}\left(\frac{C'_{\sigma,s,\alpha,u_0,\theta_0,T^*}}{[e^{C_\alpha(T^*-t)}-1]^{\frac{2\alpha-1}{3\alpha\sigma}}}\right)^k \leq \|(u,\theta)(t)\|_{\dot{H}^{s+\frac{k}{2\sigma}}}^2,\quad\forall t\in [0,T^*).
\end{align*}
As a consequence, we have the inequality below:
\begin{align}\label{wnovo1}
\frac{C_{s,\alpha,u_0,\theta_0,T^*}}{[e^{C_\alpha(T^*-t)}-1]^{\frac{2s(2\alpha-1)}{3\alpha}}}\frac{\left(\frac{2aC'_{\sigma,s,\alpha,u_0,\theta_0,T^*}}{[e^{C_\alpha(T^*-t)}-1]^{\frac{2\alpha-1}{3\alpha\sigma}}}\right)^k }{k!} &\leq  \int_{\mathbb{R}^3} \frac{(2a|\xi|^{\frac{1}{\sigma}})^k}{k!}|\xi|^{2s}|(\widehat{u},\widehat{\theta})(t)|^{2}\,d\xi,\quad\forall t\in [0,T^*).
\end{align}
%By summing over the set $\{k\in \mathbb{N};k\geq2\sigma(2\alpha-1)\}$ and applying Monotone Convergence Theorem, one obtains
%\begin{align}\label{31}
%&\frac{C_{\mu,\nu,s,\alpha,u_0,b_0}}{(T^*-t)^{\frac{2s(2\alpha-1)}{3\alpha}}}[\exp\left\{\frac{2aC'_{\mu,\nu,\sigma,s,\alpha,u_0,b_0}}{(T^*-t)^{\frac{2\alpha-1}{3\alpha\sigma}}}\right\}-\sum_{0\leq k< 2\sigma(2\alpha-1)}\frac{\left(\frac{2aC'_{\mu,\nu,\sigma,s,\alpha,u_0,b_0}}{(T^*-t)^{\frac{2\alpha-1}{3\alpha\sigma}}}\right)^k}{k!}] \leq\|(u,b)(t)\|_{\dot{H}_{a,\sigma}^s}^2,
%\end{align}
%for all $t\in [0,T^*).$
Hence, by using Monotone Convergence Theorem in the inequality (\ref{wnovo1}), we deduce
\begin{align}\label{wilber14}
\nonumber\frac{C_{s,\alpha,u_0,\theta_0,T^*}}{[e^{C_\alpha(T^*-t)}-1]^{\frac{2s(2\alpha-1)}{3\alpha}}}\sum_{k=2\sigma_0+1}^{\infty}\frac{\left(\frac{2aC'_{\sigma,s,\alpha,u_0,\theta_0,T^*}}{[e^{C_\alpha(T^*-t)}-1]^{\frac{2\alpha-1}{3\alpha\sigma}}}\right)^k }{k!} &\leq  \int_{\mathbb{R}^3} \sum_{k=2\sigma_0+1}^{\infty}\frac{(2a|\xi|^{\frac{1}{\sigma}})^k}{k!}|\xi|^{2s}|(\widehat{u},\widehat{\theta})(t)|^{2}\,d\xi\\
&\leq \|(u,\theta)(t)\|_{\dot{H}_{a,\sigma}^s}^2,
\end{align}
for all $t\in [0,T^*)$, where $2\sigma_0$ is the integer part of $2\sigma\mu$.

On the other hand, it is elementary to observe that the continuous function $f$ defined by
$$f(x)=x^{-(2\sigma_0+1)}e^{-\frac{x}{2}}\left[\displaystyle\sum_{k=2\sigma_0+1}^{\infty}\frac{x^k}{k!}\right], \quad\forall\, x>0,$$
satisfies the  inequality
\begin{align}\label{wilber13}
f(x)\geq C_{\sigma_0},\quad \forall x>0,
\end{align}
for some positive constant $C_{\sigma_0}$. Thereby, by applying  (\ref{wilber13}) to (\ref{wilber14}), one obtains
\begin{align}\label{wilber33}
\|(u,\theta)(t)\|_{\dot{H}_{a,\sigma}^s}^2&\geq
\frac{a^{2\sigma_0+1}C_{s,\alpha,\sigma,\sigma_0,u_0,\theta_0,T^*}}{[e^{C_\alpha(T^*-t)}-1]^{\frac{(2\alpha-1)[2(s\sigma+\sigma_0)+1]}{3\alpha\sigma}}}\exp\left\{\frac{aC'_{\sigma,s,\alpha,u_0,\theta_0,T^*}}{[e^{C_\alpha(T^*-t)}-1]^{\frac{2\alpha-1}{3\alpha\sigma}}}\right\},\quad\forall t\in [0,T^*).
\end{align}
This completes the proof of Corollary \ref{corolario}.

\caixa

\noindent \textbf{Data Availability Statement:} This manuscript has no associated data.

%\section*{Acknowledgments}
%\addcontentsline{toc}{section}{Acknowledgments}

%The authors are very thankful to the Reviewer
%for her/his careful analysis of our text,
%and her/his
%helpful comments which helped us to
%improve it.

%\signdm
%\signwm
%\signls

\end{document}